\documentclass{lmcs}

\usepackage{xcolor}
\usepackage{subcaption}
\definecolor{acc-green}{RGB}{0,128,0}
\definecolor{acc-red}{RGB}{219,10,91}
\definecolor{acc-blue}{RGB}{0,0,224}
\usepackage{mathtools}
\usepackage{tikz-cd}
\newcommand{\DeclareBinOperator}[2]{\newcommand{#1}{\mathbin{#2}}}
\newcommand{\mlink}[2]{\hyperref[#1]{#2}}
\newcommand{\mlabel}[1]{\phantomsection{}\label{#1}}
\newcommand{\blank}{{-}}
\newcommand{\judeq}{\mathrel{\mlink{judeq}{\doteq}}}
\newcommand{\defeq}{\mathrel{:=}}
\DeclareMathOperator{\pr}{pr}
\newcommand{\ie}{i.e.~}
\newcommand{\eg}{e.g.~}
\DeclareMathOperator{\id}{id}
\DeclareBinOperator{\comp}{\circ}
\DeclareBinOperator{\compr}{\mlink{compr}{\smash{\raisebox{-0.5pt}{$\mathrlap{\hspace{1.2pt}\lrcorner}$}}\circ}}
\DeclareBinOperator{\compl}{\mlink{compl}{\smash{\raisebox{-1.5pt}{$\mathrlap{\hspace{-1pt}\ulcorner}$}}\circ}}
\newcommand{\tr}{\mlink{tr}{{\scriptstyle\#}}}
\newcommand{\ap}{\mlink{ap}{\operatorname{ap}}}
\newcommand{\apd}{\mlink{apd}{\operatorname{apd}}}
\DeclareMathOperator{\refl}{refl}
\DeclareBinOperator{\concat}{\mlink{concat}{\operatorname{\bullet}}}
\newcommand{\htpy}{\mathrel{\mlink{htpy}{\sim}}}
\DeclareMathOperator{\reflhtpy}{refl-htpy}
\DeclareBinOperator{\hconcat}{\mlink{hconcat}{\bullet_h}}
\DeclareBinOperator{\lwhisk}{\mlink{lwhisk}{\cdot_{l}}}
\DeclareBinOperator{\rwhisk}{\mlink{rwhisk}{\cdot_{r}}}

\newcommand{\0}{\mathbf{0}}
\newcommand{\UU}{\mlink{UU}{\mathcal{U}}}
\newcommand{\T}{\mathcal{T}}
\newcommand{\N}{\mathbb{N}}
\DeclareMathOperator{\exf}{ex-falso}
\newcommand{\PO}[3]{\mlink{PO}{#2 \sqcup_{#1} #3}}
\newcommand{\inl}{\mlink{inl}{\operatorname{inl}}}
\newcommand{\inr}{\mlink{inr}{\operatorname{inr}}}
\newcommand{\glue}{\mlink{glue}{\operatorname{glue}}}
\DeclareMathOperator{\incl}{incl}
\newcommand{\depCoconePO}{\mlink{dep-cocone-po}{\operatorname{dep-cocone}}}
\newcommand{\depCoconeMapPO}{\mlink{dep-cocone-map-po}{\operatorname{dep-cocone-map}}}
\newcommand{\depCogapPO}{\mlink{dep-cogap-po}{\operatorname{dep-cogap}}}
\newcommand{\equivDD}{\mathrel{\mlink{equiv-dd}{\simeq}}}
\newcommand{\sectPO}{\mlink{sect-DD}{\operatorname{sect}}}
\newcommand{\evReflIdSystemTy}{\mlink{ev-refl-id-system}{\operatorname{ev-refl}}}
\newcommand{\evReflIdSystemDD}{\mlink{ev-refl-id-system}{\operatorname{ev-refl}}}
\DeclareMathOperator{\indidsystemDD}{ind-Q}
\newcommand{\depCoconeMapSC}{\mlink{dep-cocone-map-sc}{\operatorname{dep-cocone-map}}}
\newcommand{\inn}{\mlink{inn}{\iota}}
\newcommand{\kapn}{\mlink{kapn}{\kappa}}
\newcommand{\PA}[1]{\mlink{PAn}{P_A^{#1}}}
\newcommand{\inclA}[1]{\mlink{inclA}{\operatorname{incl}_A^{#1}}}
\newcommand{\glueA}[1]{\mlink{glueA}{\operatorname{glue}_A^{#1}}}
\newcommand{\PAdiag}[1]{\mlink{PAdiag}{P_A^{#1}}}
\newcommand{\PAinf}{\mlink{PAinf}{P_A^{\infty}}}
\newcommand{\PB}[1]{\mlink{PBn}{P_B^{#1}}}
\newcommand{\inclB}[1]{\mlink{inclB}{\operatorname{incl}_B^{#1}}}
\newcommand{\glueB}[1]{\mlink{glueB}{\operatorname{glue}_B^{#1}}}
\newcommand{\PBdiag}[1]{\mlink{PBdiag}{P_B^{#1}}}
\newcommand{\PBinf}{\mlink{PBinf}{P_B^{\infty}}}
\newcommand{\inA}[1]{\mlink{inA}{\iota_A^{#1}}}
\newcommand{\inB}[1]{\mlink{inB}{\iota_B^{#1}}}
\newcommand{\kapA}[1]{\mlink{kapA}{\kappa_A^{#1}}}
\newcommand{\kapB}[1]{\mlink{kapB}{\kappa_B^{#1}}}
\newcommand{\concatpre}{\mathbin{\bullet}}
\newcommand{\concats}[1]{\mathbin{\mlink{concats}{\bullet_{#1}}}}
\newcommand{\concatinvs}[1]{\mathbin{\mlink{concatinvs}{\bullet_{#1}}}}
\DeclareBinOperator{\concatinf}{\mlink{concatinf}{\bullet_{\infty}}}
\newcommand{\PhiDiag}[1]{\mlink{PhiDiag}{\Phi_{#1}}}
\newcommand{\dA}[1]{\mlink{dA}{d_A^{#1}}}
\newcommand{\dB}[1]{\mlink{dB}{d_B^{#1}}}
\newcommand{\tAn}[1]{\mlink{tAn}{t_A^{#1}}}
\newcommand{\tBn}[1]{\mlink{tBn}{t_B^{#1}}}
\newcommand{\tAinf}{\mlink{tAinf}{t_A^{\infty}}}
\newcommand{\tBinf}{\mlink{tBinf}{t_B^{\infty}}}
\newcommand{\tSinf}{\mlink{tSinf}{t_S^{\infty}}}
\newcommand{\reflinf}{\mlink{refl-infty}{\operatorname{refl}_{\infty}}}
\date{}
\title{}
\begin{document}

\title[Formalization of the zigzag construction of path spaces of pushouts]%
{Formalization of the zigzag construction of path spaces of pushouts in homotopy type theory}

\renewcommand{\sectionautorefname}{Section}

\newcommand{\ExternalLink}{%
\tikz[x=1.2ex, y=1.2ex, baseline=-0.05ex]{%
    \begin{scope}[x=1ex, y=1ex]
        \clip (-0.1,-0.1)
            --++ (-0, 1.2)
            --++ (0.6, 0)
            --++ (0, -0.6)
            --++ (0.6, 0)
            --++ (0, -1);
        \path[draw,
            line width = 0.5,
            rounded corners=0.5]
            (0,0) rectangle (1,1);
    \end{scope}
    \path[draw, line width = 0.5] (0.5, 0.5)
        -- (1, 1);
    \path[draw, line width = 0.5] (0.6, 1)
        -- (1, 1) -- (1, 0.6);
    }%
}

\captionsetup{subrefformat=parens}
\subcaptionsetup[figure]{labelfont=rm}

\DeclareRobustCommand{\constrEnd}{%
  \leavevmode\unskip\penalty9999 \hbox{}\nobreak\hfill
  \quad\hbox{$\lrcorner$}%
}
\makeatletter
\def\@endtheorem{\constrEnd\endtrivlist\@endpefalse}
\makeatother

\theoremstyle{definition}\newtheorem{constr}[thm]{Construction}


\author[V.~Štěpančík]{Vojtěch Štěpančík\lmcsorcid{0009-0008-6749-5665}}
\address{Nantes Université, École Centrale Nantes, CNRS, INRIA, LS2N, UMR 6004, France}
\email{vojtech.stepancik@inria.fr}

\begin{abstract}
A pre-print of Wärn gives a pen-and-paper construction of a type family characterizing the path spaces of an arbitrary pushout, and a natural language argument for its correctness. This paper presents the first formalization of the construction and a proof that it is fiberwise equivalent to the path spaces. The formalization is carried out in axiomatic homotopy type theory, using the Agda proof assistant and the agda-unimath library.
\end{abstract}

\maketitle
\section{Introduction}
\label{sec:org54d9ae1}

Synthetic homotopy theory is a branch of type theory which treats types as homotopy spaces, with elements representing points, and identity types representing path spaces \cite{AW09}. An important discipline of synthetic homotopy theory is the study of path spaces, and in order to say anything meaningful about the path spaces of a type \(A\), the first step is often to construct a type family \(P\) that is convenient to work with, and show that there is an equivalence \((a_0 = a) \simeq P(a)\) for all \(a : A\). A good choice of \(P\) can help proving results about truncatedness and connectivity of \(A\), \eg putting bounds on its non-trivial fundamental groups.

A pre-print of Wärn \cite{War23} gives a pen-and-paper construction of such a type family for pushouts, including a proof that it is equivalent to the path spaces. It may be seen as a generalization of the James construction as defined in homotopy type theory \cite{Bru19}, which describes loop spaces of pointed connected types.

This paper presents the first formalization of Wärn's construction and its correctness. It is carried out in axiomatic homotopy type theory \cite{UF13,Rij22}, using the Agda proof assistant \cite{Agda} and the agda-unimath library \cite{AU25}. In the process, the author discovered that the original concise two-line proof relies implicitly on functoriality of sequential colimits with respect to dependent and fiberwise maps. We develop those, noting that their construction involves technical path algebra, and fill in the details of the coherences necessary to finish the proof of correctness.
\subsection*{Overview}
\label{sec:orgf11ddfa}
After \autoref{sec:hott} fixes the foundations of the formalization and notations of the paper, we dedicate \autoref{sec:diagrams} to explaining the diagrammatic language used in the paper. We then introduce pushouts in \autoref{sec:pushouts}, which contains original work on identity systems of descent data. In \autoref{sec:seqcol} we recall sequential colimits and develop new dependent functoriality principles. Finally in \autoref{sec:idpo} we present an encoding of the zigzag construction accepted by Agda's termination checker, and show that it forms an identity system. All definitions and proofs come with a link to a fixed online version of their formalization, indicated by the ``\ExternalLink{}'' icon. Names of defined constants in this document also serve as internal links to their respective definitions.
\subsection*{Contributions}
\label{sec:orge86a1f0}
The main theoretical contribution is the problem decomposition of the ``invisible mathematics''\footnote{A term coined, to the best of my knowledge, by Andrej Bauer in his talk ``Formalizing invisible mathematics.'' \url{https://math.andrej.com/2023/02/13/formalizing-invisible-mathematics/}} in the proof of correctness of the zigzag construction: introducing identity systems of descent data to characterize path spaces of pushouts; defining dependent and fiberwise morphisms of sequential diagrams, their compositions and preservation of these compositions when taking the colimit; leveraging these functoriality results to construct a necessary coherence datum; and providing diagrammatic explanations of the coherence data involved. In order to accurately represent diagrammatic reasoning, an extension of a standard visual language of commuting diagrams was developed.

On the formalization side, all code starting with \autoref{lem:sect-sect-descent-data} was contributed to the library by the author, except for \autoref{defn:identity-system}. Sequential colimits had been formalized in Lean 3 \cite{SvDR20}, but with a different objective, so the overlap with the present work consists of definitions. No infrastructure for sequential colimits had been available in the agda-unimath library. The formalization also extensively leverages descent data as a code organization tool, an approach illustrated for the circle by Rijke \cite[Section 22.2]{Rij22}, but not found in other formalized libraries.
\section{Homotopy Type Theory}
\label{sec:hott}
We assume some familiarity with homotopy type theory (HoTT) and common vocabulary, such as type families, identity types, equivalences, homotopies, fiberwise maps, commuting diagrams and univalence. The foundational framework of this paper is axiomatic homotopy type theory as described by Rijke in \cite{Rij22} and implemented in the agda-unimath library \cite{AU25}. The main difference from ``Book HoTT'' \cite{UF13} is that the formalization does not have first-class higher inductive types (HITs) --- in keeping with the library style, those are either postulated without rewrite rules, or a universal cocone is accepted as a function argument. As a result, functions defined by the universal property only come with propositional computation rules, while some of them would be judgmental in Book HoTT.

We use the symbol ``\(\judeq\)''\mlabel{judeq} to denote metatheoretical judgmental equality, ``\(\defeq\)'' for definitions, and ``\(=\)'' for the identity type. Elements of identity types are called ``paths'' or ``identifications''. We use the symbol ``\(\UU\)''\mlabel{UU} for univalent universes. Universe levels are not specified in the text, but the formalization is appropriately universe polymorphic. We adopt Agda's notation for dependent function types and implicit arguments: the type of dependent functions from a type \(A\) to a type family \(P\) over \(A\) is denoted \((a : A) \to P(a)\). Both juxtaposition, \eg \(f x\), and explicit parentheses, \eg \(f(x)\), are used for function application. Implicit arguments are put in curly braces, as \(\{a : A\} \to P(a)\), and are omitted when applying the function, inferring the appropriate value from the surrounding context. When declaring types of function symbols, we also use the shorter notation \(e(a : A): P(a)\) instead of \(e : (a : A) \to P(a)\). Dependent pair types are denoted \(\Sigma (a : A).\, P(a)\) with projections \(\pr_i\) in diagrams, but in writing we prefer to write them out as ``the type of pairs \((a, p)\) with \(a : A\) and \(p : P(a)\)''. We implicitly use the structure identity principle \cite[Section 11.6]{Rij22} to characterize path spaces of iterated sigma types component-wise.

Paths are concatenated in diagrammatic order with the left associative ``\(\concat\)''\mlabel{concat} operation, and inverted with the \((\blank)^{-1}\) operation. A path \(p : x = y\) in a type \(A\) induces the transport function \(p \tr : P(x) \to P(y)\)\mlabel{tr} by sending \(\refl\) to the identity. A function \(f : A \to B\) acts on paths in \(A\) by the operation \(\ap_f : x = y \to f(x) = f(y)\)\mlabel{ap}, and a dependent function \(s : (a : A) \to P(a)\) acts on paths in \(A\) by the operation \(\apd_s(p : x = y) : p\tr s(x) = s(y)\)\mlabel{apd}. We call paths of the form \(p \tr u = v\) ``dependent paths from \(u\) to \(v\) over \(p\)''. A homotopy \(H : f \htpy g\)\mlabel{htpy} between functions \(f, g : (a : A) \to P(a)\) is a family of paths \(H(a) : f(a) = g(a)\). Homotopies compose vertically with the ``\(\hconcat\)''\mlabel{hconcat} operator, and may be extended by the whiskering operations, namely left whiskering \(j \lwhisk \blank : (f \htpy g) \to (j \comp f) \htpy (j \comp g)\)\mlabel{lwhisk} and right whiskering \(\blank \rwhisk k : (f \htpy g) \to (f \comp k) \htpy (g \comp k)\)\mlabel{rwhisk} for suitably typed maps \(j\) and \(k\). We allow ourselves to abuse notation so that for a homotopy \(H : f \htpy g\), the fiberwise transport \((a \mapsto (H(a) \tr)) : (a : A) \to P(f a) \to P(g a)\) is written as just \((H \tr)\).
\section{Commuting shapes and coherences}
\label{sec:diagrams}
In the latter parts of the paper we use diagrams where possible to explain constructions and theorems. Two kinds of diagrams are present: non-dependent ones, which express commutativity of non-dependent functions between types as homotopies, and commutativity of such homotopies; and dependent ones, which can express diagrams involving type families, fiberwise maps, dependent functions, their homotopies and homotopies of those homotopies, in a limited capacity.

Non-dependent diagrams are well known. We use the following to represent homotopies \(f \htpy g\), triangles \(g \comp f \htpy h\), and squares \(g \comp f \htpy k \comp h\):
\begin{center}
\begin{tikzcd}
A \arrow[r, bend left, "g"] \arrow[r, bend right, "f"'] & B
\end{tikzcd}
\hspace{2em}
\begin{tikzcd}[column sep=small, row sep=small]
A \arrow[rr, "h"] \arrow[dr, "f"'] && C \\
& B \arrow[ur, "g"']
\end{tikzcd}
\hspace{2em}
\begin{tikzcd}[sep=15pt]
A \arrow[r, "h"] \arrow[d, "f"'] & C \arrow[d, "k"] \\
B \arrow[r, "g"'] & D.
\end{tikzcd}
\end{center}
If dictated by context they may represent homotopies in the opposite direction, \ie \(g \htpy f\), \(h \htpy g \comp f\) and \(k \comp h \htpy g \comp f\).

One dimension higher, we talk about commuting (triangular) prisms
\begin{center}
\begin{tikzcd}[row sep=small]
  A'
  \arrow[dd, "h_A"']
  \arrow[rr, "f'"]
  \arrow[dr, "h'"', near end]
  && B'
  \arrow[dd, "h_B"]
  \arrow[dl, "g'", near end] \\
  & C' \\
  A
  \arrow[rr, "f", near start]
  \arrow[dr, "h"']
  && B,
  \arrow[dl, "g"] \\
  & C
  \arrow[from=uu, "h_C", near start, crossing over]
\end{tikzcd}
\end{center}
which are fillers of shapes composed of three squares \(L : h_C \comp h' \htpy h \comp h_A\), \(R : h_C \comp g' \htpy g \comp h_B\), \(H : f \comp h_A \htpy h_B \comp f'\) and two triangles \(U : g' \comp f' \htpy h'\) and \(D : g \comp f \htpy h\). The prism is then an element of the type \(L \hconcat ((D \rwhisk h_A) \hconcat (g \lwhisk H)) \htpy (h_C \lwhisk U) \hconcat (R \rwhisk f')\). It represents a proof that the two homotopies of type \(h_C \comp h' \htpy g \comp h_B \comp f'\), namely one composed of \(L\), \(D\) and \(H\), and the other of \(U\) and \(R\), are themselves homotopic.

Dependent diagrams are less standardized. In this paper we keep the convention that dependent diagrams are indicated by either containing downward facing arrows with two heads, or upwards facing arrows, unless stated otherwise. Double headed arrows represent type families --- a pair of type families \(P : A \to \UU\) and \(Q : B \to \UU\) would be drawn as
\begin{center}
\begin{tikzcd}
P \arrow[d, two heads] & Q \arrow[d, two heads] \\
A & B.
\end{tikzcd}
\end{center}
We can add horizontal arrows to the picture: an arrow at the bottom is a regular function between types \(f : A \to B\), and an arrow at the top lying over \(f\) represents a fiberwise function \(e\{a : A\}: P(a) \to Q(f(a))\).
\begin{center}
\begin{tikzcd}
P \arrow[d, two heads] \arrow[r, "e"] & Q \arrow[d, two heads] \\
A \arrow[r, "f"'] & B.
\end{tikzcd}
\end{center}
Notice two things: the codomain of \(e\) is drawn as \(Q\), not \(Q \comp f\), since the fact that \(e\) is over \(f\) is represented visually, and \(Q \comp f\) is not a type family over \(B\); and this square does not represent an inhabitable type, it just asserts the types of \(P\), \(Q\), \(f\) and \(e\).

We may also have two type families over the same type, in which case a fiberwise map
\begin{center}
\begin{tikzcd}[row sep=small]
P \arrow[dr, two heads] \arrow[rr, "e"] && Q \arrow[dl, two heads] \\
& A
\end{tikzcd}
\end{center}
is treated as lying over \(\id\), so it has the type \(e\{a\}: P(a) \to Q(a)\).

Adding another dimension, we can talk about homotopies over a homotopy \(H : f \htpy g\), or dependent homotopies:
\begin{center}
\begin{tikzcd}[row sep=25pt]
P \arrow[d, two heads] \arrow[r, bend left, "h"] \arrow[r, bend right, "e"' pos=0.4, "/"{marking, pos=0.7}] & Q \arrow[d, two heads] \\
A \arrow[r, bend left, "g"] \arrow[r, bend right, "f"'] & B
\end{tikzcd}
\hspace{20pt}
\begin{tikzcd}[row sep=25pt]
P \arrow[d, two heads] \arrow[r, bend left, "e" pos=0.4, "/"{marking, pos=0.7}] \arrow[r, bend right, "h"'] & Q \arrow[d, two heads] \\
A \arrow[r, bend left, "f"] \arrow[r, bend right, "g"'] & B.
\end{tikzcd}
\end{center}
Note the slashed arrows: the map \(e\{a\} : P(a) \to Q(f(a))\) is over \(f\) and \({h\{a\} : P(a) \to Q(g(a))}\) is over \(g\), so they cannot be homotopic, as they have different codomains. The slashed arrow indicates that we put an implicit transport along \(H\) after \(e\). In other words, a homotopy between \(e\) and \(h\) over \(H\) is a family of dependent paths \(H'\{a\}(p : P(a)): H(a) \tr e(p) = h(p)\). Depending on the orientation of the bottom homotopy, the slashed arrow may be on either end of the top homotopy. Note that these diagrams still only assert types of the top layer, so they scale seamlessly to dependent triangles and dependent squares, which look like prisms and cubes, but do not carry any proper coherence information. The slashed arrow notation was chosen to be minimally visually disruptive, since the transports are usually uninteresting for intuition, but the author wants them to be present in order to be faithful to the implementation.

If we want to add dependent functions to a diagram, we draw them pointing upwards. We can draw a pair of dependent functions \(s : (a : A) \to P(a)\) and \(t : (b : B) \to Q(b)\) as
\begin{center}
\begin{tikzcd}
P & Q \\
A \arrow[u, "s" pos=0.4] & B. \arrow[u, "t"' pos=0.4]
\end{tikzcd}
\end{center}
The only difference between dependent and non-dependent functions is the direction in which they point. In particular dependent functions do not have a binder \((a : A)\) and the codomain is not applied like \(P(a)\). If a diagram contains dependent functions we do not draw the double headed arrows, because bases of type families are clear from their layout.

We may once again add horizontal arrows, which gets us the diagram
\begin{center}
\begin{tikzcd}
P \arrow[r, "e"] & Q \\
A \arrow[u, "s" pos=0.4] \arrow[r, "f"'] & B, \arrow[u, "t"' pos=0.4]
\end{tikzcd}
\end{center}
where \(f : A \to B\) is a function and \(e\{a\} : P(a) \to Q(f(a))\) is a fiberwise function. This diagram does represent a type of homotopies, namely \(e \comp s \htpy t \comp f\) in the dependent function type \((a : A) \to Q(f(a))\). We call such homotopies ``commuting squares of sections''.

We may add a third dimension by considering homotopies of squares of sections which share the bottom and top maps, \eg consider the two squares on the left
\begin{center}
\begin{tikzcd}
P \arrow[r, "e"] & Q \\
A \arrow[u, "s" pos=0.45] \arrow[r, "f"'] \arrow[ur, phantom, "F"] & B \arrow[u, "t"' pos=0.4]
\end{tikzcd}
\hspace{10pt}
\begin{tikzcd}
P \arrow[r, "e"] & Q \\
A \arrow[u, "s'" pos=0.45] \arrow[r, "f"'] \arrow[ur, phantom, "N"] & B \arrow[u, "t'"' pos=0.4]
\end{tikzcd}
\hspace{25pt}
\begin{tikzcd}[column sep=60pt, row sep=25pt, cramped]
  P \arrow[r, "e"]
  \arrow[r, phantom, "F" pos=0.35, yshift=-12pt] & Q \\
  A \arrow[r, "f"'] \arrow[u, "s", bend left=45, ""{name=S0, anchor=center}] \arrow[u, "s'"', bend right=45, ""{name=T0, anchor=center}]
  \arrow[r, phantom, "N" pos=0.65, yshift=12pt] & B, \arrow[u, "t", bend left=45, ""{name=S1, anchor=center}] \arrow[u, "t'"', bend right=45, ""{name=T1, anchor=center}]
  \arrow[from=S0, to=T0, phantom, "L"]
  \arrow[from=S1, to=T1, phantom, "R"]
\end{tikzcd}
\end{center}
then a homotopy between them is a pair of homotopies \(L : s \htpy s'\) and \(R : t \htpy t'\), together with a coherence \(F \hconcat (R \rwhisk f) \htpy (e \lwhisk L) \hconcat N\), visualized as the cylinder on the right above.

Orthogonally, we may add the last dimension by considering homotopies and dependent homotopies. In that case we have two squares of sections on the left
\begin{center}
\begin{tikzcd}
P \arrow[r, "e"] & Q \\
A \arrow[u, "s" pos=0.4] \arrow[r, "f"'] \arrow[ur, phantom, "N"] & B \arrow[u, "t"' pos=0.4]
\end{tikzcd}
\hspace{10pt}
\begin{tikzcd}
P \arrow[r, "h"] & Q \\
A \arrow[u, "s" pos=0.4] \arrow[r, "g"'] \arrow[ur, phantom, "F"] & B \arrow[u, "t"' pos=0.4]
\end{tikzcd}
\hspace{25pt}
\begin{tikzcd}[row sep=25pt]
P \arrow[r, bend left, "h"] \arrow[r, bend right, "e"' pos=0.4, "/"{marking, pos=0.7}] & Q \\
A \arrow[r, bend left, "g"] \arrow[u, "s" pos=0.4] \arrow[r, bend right, "f"'] & B, \arrow[u, "t" pos=0.4]
\end{tikzcd}
\end{center}
a homotopy \(H : f \htpy g\) and a dependent homotopy \(H' : (H \tr) \comp e \htpy h\). The ``cylinder of sections'' is drawn as the cylinder on the right above, and represents the type \((a : A) \to H'(s(a)) \concat F(a) = \ap_{H(a) \tr}(N(a)) \concat \apd_t(H(a))\).

Cylinders of sections extend to prisms of sections and cubes of sections, but they are not judgmentally the same types. For example the prism of sections
\begin{center}
\begin{tikzcd}[row sep=small]
  & Q \arrow[dr, "g'"] \arrow[from=dd, "s", very near start] \\
  P \arrow[rr, crossing over, "/"{marking, near end}, "h'" near start]
  \arrow[ur, "f'"]
  & & R \\
  & B \arrow[dr, "g" pos=0.4] \\
  A \arrow[rr, "h"'] \arrow[uu, "t"]
  \arrow[ur, "f" pos=0.4]
  & & C \arrow[uu, "u"']
\end{tikzcd}
\end{center}
with bottom triangle \(H : h \htpy g \comp f\), top dependent triangle \(H' : (H\tr) \comp h' \htpy g' \comp f'\), left square \(L : f' \comp t \htpy s \comp f\), right square \(R : g' \comp s \htpy u \comp g\) and front square \(F : h' \comp t \htpy u \comp h\), represents the type of coherences \((a : A) \to H'(t(a)) \concat \ap_{g'}(L(a)) \concat R(f(a)) = \ap_{H(a)\tr}(F(a)) \concat \apd_u(H(a))\),
which needs associativity of path concatenation to have the type of a cylinder of sections.

The last shape we consider are cubes of sections
\begin{center}
\begin{tikzcd}[cramped]
  P \arrow[dr, "h'"'] \arrow[rr, crossing over, "f'"]
  & & Q \arrow[from=dd, "t" near start]
  \arrow[dr, "g'"] \\[-5pt]
  & R
  \arrow[rr, "/" {marking, near end}, crossing over, "k'" pos=0.3]
  & & S \\
  A \arrow[uu, "s"] \arrow[dr, "h"'] \arrow[rr, "f", pos=0.8]
  & & B \arrow[dr, "g"]
  \\[-5pt]
  & C \arrow[uu, crossing over, "u"' pos=0.2] \arrow[rr, "k"']
  & & D \arrow[uu, crossing over, "v"']
\end{tikzcd}
\end{center}
with bottom square \(H\), top dependent square \(H'\), left square \(L\), right square \(R\), far square \(F\) and near square \(N\). They represent the type
\begin{align*}
  (a : A) \to{}& H'(s(a)) \concat \ap_{g'}(F(a)) \concat R(f(a)) \\
  &= \ap_{(H(p) \tr) \comp k'}(L(a)) \concat \ap_{H(p) \tr}(N(h(a))) \concat \apd_v(H(a)),
\end{align*}
which is again equivalent to the type of cylinders of sections if we glued together the pairs of faces \(L, N\) and \(F, R\), but it is not judgmentally equal to it.

An alternative to this ``indexed'' encoding was considered, namely the ``displayed'' encoding where type families \(P : A \to \UU\) are replaced by types over \(A\), \ie a type \(\widetilde{P}\) and a map \(p : \widetilde{P} \to A\), and dependent maps \(s : (a : A) \to P(a)\) are replaced by maps \(\widetilde{s} : A \to \widetilde{P}\) equipped with a homotopy \(r : p \comp s \htpy \id\). These two encodings are equivalent by the (un)straigtening construction \cite{HTT}, and offer different trade-offs. The displayed encoding allows for precise diagrammatic reasoning, because dependent concepts are explained in terms of non-dependent ones, so all types and maps in diagrams become non-dependent. This should in theory also allow more code reuse in the formalization between dependent and non-dependent infrastructure. The disadvantage is that we miss out on some judgmental equalities --- unstraightening a dependent map \(s : (a : A) \to P(a)\) gives the map \((a \mapsto (a, s(a))) : A \to \Sigma (a : A). P(a)\), which composes with the retraction \(p : \Sigma (a : A). P(a) \to A\) to the identity map up to judgmental equality. However we cannot abstract over judgmental equalities, so infrastructure must abstract over arbitrary homotopies, which leads to proliferation of path algebra in library code. In a previous, unsuccessful attempt at proving the zigzag construction correct the author started building unstraightened infrastructure, but it also ended up having noticeably worse performance than the straightened one, though he did not investigate this in detail. Furthermore, unstraightening a diagram raises its dimension by one. For example, an unstraightening of a square of sections consists of the data of a commuting prism: the square showing that the top map lies over the bottom becomes a bona fide homotopy \(f \comp p \htpy q \comp \widetilde{e}\), and in addition to giving a commuting square \(\widetilde{e} \comp \widetilde{s} \htpy \widetilde{t} \comp f\) we need to prove that the two squares compose to \(\reflhtpy : f \comp \id \htpy \id \comp f\). Similarly, a cube of sections is represented by a certain hypercube. This data then needs to be taken into account in functoriality principles.

For those reasons the author chose to formalize the indexed infrastructure, and add dependent vocabulary to the diagram language.
\section{Pushouts}
\label{sec:pushouts}
Pushouts are colimits specified by span diagrams
\begin{tikzcd}[cramped]
A & S \arrow[l, "f"'] \arrow[r, "g"] & B.
\end{tikzcd}
In other words, given such a span diagram, its pushout is a type \(X\) with two point constructors, \(\inl : A \to X\) and \(\inr : B \to X\), and a path constructor \(\glue(s : S) : \inl(f(s)) = \inr(g(s))\). This description can be used directly to define pushouts in type theories with first-class HITs. The flavor of Agda used in the formalization does not support first-class HITs, so instead we define pushouts to be structures satisfying a certain universal property, which gives us an induction principle for them. When interpreting types as homotopy spaces, we may think of the pushout as two distinct components corresponding to copies of \(A\) and \(B\), to which we add (higher) paths between \(f(s)\) and \(g(s)\) for each \(s : S\).

Except for identity systems at the end of the section, we follow Rijke's development of descent for pushouts \cite[Chapter 2]{Rij19}. We summarize the definitions we use below.

\begin{defi}
Consider a span diagram
\begin{tikzcd}[cramped]
A & S \arrow[l, "f"'] \arrow[r, "g"] & B.
\end{tikzcd}

\begin{enumerate}[i.]
\item A \textbf{cocone} \href{https://archive.vojtechstep.eu/zigzag-construction/synthetic-homotopy-theory.cocones-under-spans.html\#cocones}{\ExternalLink} to a type \(X\) is a triple \((i, j, H)\), which consist of maps \(i : A \to X\) and \(j : B \to X\), and a homotopy \(H : i \comp f \htpy j \comp g\).
\item A \textbf{dependent cocone} \href{https://archive.vojtechstep.eu/zigzag-construction/synthetic-homotopy-theory.dependent-cocones-under-spans.html\#dependent-cocones}{\ExternalLink} to a type family \(P : X \to \UU\) over a cocone \((i, j, H)\) to \(X\) is a triple \((i', j', H')\), where \(i' : (a : A) \to P(i(a))\) and \(j' : (b : B) \to P(j(b))\) are dependent maps, and \(H'(s : S) : H(s)\tr i'(f(s)) = j'(g s)\) is a dependent homotopy. We denote the type of dependent cocones over a cocone \(c\) to a type family \(P\) by \(\depCoconePO(c, P)\)\mlabel{dep-cocone-po}.
\item The \textbf{cocone map} \href{https://archive.vojtechstep.eu/zigzag-construction/synthetic-homotopy-theory.cocones-under-spans.html\#postcomposing-cocones-under-spans-with-maps}{\ExternalLink} accepts a cocone \((i, j, H)\) to \(X\) and a function \(h : X \to Y\), and constructs the cocone \((h \comp i, h \comp j, s \mapsto \ap_h(H(s)))\) to \(Y\).
\item The \textbf{dependent cocone map} \href{https://archive.vojtechstep.eu/zigzag-construction/synthetic-homotopy-theory.dependent-cocones-under-spans.html\#postcomposing-dependent-cocones-with-maps}{\ExternalLink} accepts a cocone \((i, j, H)\) to \(X\) and a dependent map \(t (x : X): P(x)\), and constructs the dependent cocone \((t \comp i, t \comp j, s \mapsto \apd_t(H(s)))\) to \(P\). We denote the dependent cocone map by \(\depCoconeMapPO\)\mlabel{dep-cocone-map-po}.
\item A cocone \(c\) to \(X\) is a \textbf{pushout} if either its cocone map is an equivalence for all \(Y\) (the \textbf{universal property} \href{https://archive.vojtechstep.eu/zigzag-construction/synthetic-homotopy-theory.universal-property-pushouts.html\#the-universal-property-of-pushouts-1}{\ExternalLink}), or its dependent cocone map is an equivalence for all \(P\) (the \textbf{dependent universal property} \href{https://archive.vojtechstep.eu/zigzag-construction/synthetic-homotopy-theory.dependent-universal-property-pushouts.html\#the-dependent-universal-property-of-pushouts-1}{\ExternalLink}). In that case, we use the name \(\PO{S}{A}{B}\)\mlabel{PO} for \(X\), the names \(\inl\)\mlabel{inl} (``left point constructor'') and \(\inr\)\mlabel{inr} (``right point constructor'') for \(i\) and \(j\), and \(\glue\)\mlabel{glue} (``path constructor'') for \(H\). We refer to both \(X\) and the cocone as ``pushout''.
\item The cocone map of a pushout \(\PO{S}{A}{B}\) has an inverse which sends a cocone \(c\) to \(Y\) to a map \(\PO{S}{A}{B} \to Y\). This is called the \textbf{cogap map}~\href{https://archive.vojtechstep.eu/zigzag-construction/synthetic-homotopy-theory.universal-property-pushouts.html\#the-cogap-map}{\ExternalLink} of \(c\).
\item Similarly, the dependent map corresponding to a dependent cocone \(d\) to \(P\) is called the \textbf{dependent cogap map} \href{https://archive.vojtechstep.eu/zigzag-construction/synthetic-homotopy-theory.dependent-universal-property-pushouts.html\#the-dependent-cogap-map}{\ExternalLink} of \(d\), denoted \(\depCogapPO(d)\)\mlabel{dep-cogap-po}.
\item The type of \textbf{descent data} \href{https://archive.vojtechstep.eu/zigzag-construction/synthetic-homotopy-theory.descent-data-pushouts.html\#descent-data-for-pushouts-1}{\ExternalLink} is the type of triples \((P_A, P_B, P_S)\), with \(P_A : A \to \UU\), \(P_B : B \to \UU\) type families, and \(P_S\{s : S\} : P_A(f s) \simeq P_B(g s)\) a family of equivalences.
\item The type of \textbf{sections} \href{https://archive.vojtechstep.eu/zigzag-construction/synthetic-homotopy-theory.sections-descent-data-pushouts.html\#sections-of-descent-data-for-pushouts-1}{\ExternalLink} of descent data \((P_A, P_B, P_S)\), denoted \(\sectPO(P_A, P_B, P_S)\)\mlabel{sect-DD}, is the type of triples \((t_A, t_B, t_S)\), where \(t_A : (a : A) \to P_A(a)\) and \(t_B : (b : B) \to P_B(b)\) are dependent functions, and \(t_S(s) : P_S(t_A(f s)) = t_B(g s)\) is a homotopy.
\item The \textbf{total span diagram} \href{https://archive.vojtechstep.eu/zigzag-construction/synthetic-homotopy-theory.flattening-lemma-pushouts.html\#the-total-span-diagram-of-descent-data}{\ExternalLink} of descent data \((P_A, P_B, P_S)\) is the diagram\\
\begin{tikzcd}
\Sigma (a : A). P_A(a) & \Sigma (s : S). P_A(f s) \arrow[l, "\Sigma (f). \id"' outer sep=3pt] \arrow[r, "\Sigma (g). P_S" outer sep=3pt] & \Sigma (b : B). P_B(b)
\end{tikzcd}
\item For descent data \((P_A, P_B, P_S)\) and \((R_A, R_B, R_S)\), an \textbf{equivalence of descent data} \href{https://archive.vojtechstep.eu/zigzag-construction/synthetic-homotopy-theory.equivalences-descent-data-pushouts.html\#equivalences-of-descent-data-for-pushouts-1}{\ExternalLink} is a triple \((e_A, e_B, e_S)\), where \(e_A\{a\} : P_A(a) \simeq R_A(a)\) and \(e_B\{b\} : P_B(b) \simeq R_B(b)\) are fiberwise equivalences, and \(e_S\{s\} : e_B\{gs\} \comp P_S\{s\} \htpy R_S\{s\} \comp e_A\{fs\}\) is a family of commuting squares. We denote the type of equivalences by \((P_A, P_B, P_S) \equivDD (R_A, R_B, R_S)\)\mlabel{equiv-dd}.
\end{enumerate}
\end{defi}

Defining pushouts in terms of either the universal or the dependent universal property is justified, since the two propositions are equivalent.

\begin{lemC}[{\cite[Proposition 2.1.6]{Rij19}}]
\href{https://archive.vojtechstep.eu/zigzag-construction/synthetic-homotopy-theory.dependent-universal-property-pushouts.html\#the-nondependent-and-dependent-universal-property-of-pushouts-are-logically-equivalent}{\ExternalLink} A cocone satisfies the universal property of pushouts if and only if it satisfies the dependent universal property.
\end{lemC}

We include descent data in the definitions since it is the main approach to working with type families over pushouts in the formalization. Instead of treating descent as a subject of theoretical study, we use it as a practical tool for code organization. This technique is inspired by Rijke \cite[Section 22.2]{Rij22}. It is facilitated by the descent theorem, which establishes a one-to-one correspondence between type families over a pushout and descent data over its defining span.

\begin{thmC}[{\cite[Proposition 2.2.2]{Rij19}} (Descent)]
\href{https://archive.vojtechstep.eu/zigzag-construction/synthetic-homotopy-theory.descent-property-pushouts.html\#theorem}{\ExternalLink} The map taking a type family \(P\) over \(\PO{S}{A}{B}\) to the descent data \((P \comp \inl, P \comp \inr, (\glue \tr))\) is an equivalence.
\end{thmC}

The theory of descent tells us that to study behavior over pushouts, it suffices to study behavior over its two components which is in a sense ``coherent'' over the overlaps induced by \(S\). In the case of type families, the behavior on components is captured by \(P_A\) and \(P_B\), while the coherence \(P_S\) ensures that \(P_A\) and \(P_B\) ``behave the same'' (are equivalent) when restricted to the points \(f(s)\) and \(g(s)\), respectively, connected by \(\glue(s)\). For an equivalence of descent data, the coherence is \(e_S\), which states that the fiberwise equivalences \(e_A\), \(e_B\) are compatible with the transition maps \(P_S\), \(R_S\). With some setup, we will show that equivalences of descent data correspond in a precise sense to fiberwise equivalences of the corresponding type families.

\begin{lem}
\href{https://archive.vojtechstep.eu/zigzag-construction/synthetic-homotopy-theory.sections-descent-data-pushouts.html\#sections-of-families-over-a-pushout-correspond-to-sections-of-the-corresponding-descent-data}{\ExternalLink} Consider a type family \(P : \PO{S}{A}{B} \to \UU\), descent data \((P_A, P_B, P_S)\), and an equivalence of descent data \((e_A, e_B, e_S)\) between \((P \comp \inl, P \comp \inr, (\glue \tr))\) and \((P_A, P_B, P_S)\). Then there is an equivalence between the type of dependent maps \((x : \PO{S}{A}{B}) \to P(x)\) and the sections \(\sectPO(P_A, P_B, P_S)\).
\label{lem:sect-sect-descent-data}
\end{lem}

\begin{proof}
Construct the map by taking a dependent map \(h : (x : \PO{S}{A}{B}) \to P(x)\) to the triple \((h_A, h_B, h_S)\), where \(h_A(a : A) \defeq e_A(h(\inl(a)))\), similarly \(h_B(b : B) \defeq e_B(h(\inr(b)))\). Then \(h_S(s : S)\) is supposed to be an identification \(P_S(e_A(h(\inl(f s)))) = e_B(h(\inr(g s)))\), and we obtain it by composing \(e_S(h(\inl(f s)))^{-1} : P_S(e_A(h(\inl(f s)))) = e_B(\glue(s) \tr h(\inl(f s)))\) and \(\ap_{e_B}(\apd_h(\glue(s))) : e_B(\glue(s) \tr h(\inl(f s))) = e_B(h(\inr(g s)))\).

We factor this map as
\(((x : \PO{S}{A}{B}) \to P(x)) \simeq \depCoconePO(\PO{S}{A}{B}, P) \simeq \sectPO(P_A, P_B, P_S)\), where the first map is the dependent cocone map, and the second takes a triple \((i', j', H')\) to \((e_A \comp i', e_B \comp j', s \mapsto e_S(i'(f(s))))^{-1} \concat \ap_{e_B}(H'(s))\). Since \(\PO{S}{A}{B}\) is the pushout, we have that the first map is an equivalence. The second map acts by equivalences component-wise: postcomposing by the equivalences \(e_A\) and \(e_B\) is an equivalence, left whiskering by the equivalence \(e_S\) is an equivalence, and concatenating with an identification is an equivalence. It follows that the second map is an equivalence as well, so the composite is an equivalence.
\end{proof}

\begin{lem}
\href{https://archive.vojtechstep.eu/zigzag-construction/synthetic-homotopy-theory.descent-data-equivalence-types-over-pushouts.html\#sections-of-descent-data-for-families-of-equivalences-correspond-to-equivalences-of-descent-data}{\ExternalLink} Consider type families \(P\) and \(R\) over the pushout \(\PO{S}{A}{B}\), descent data \((P_A, P_B, P_S)\) and \((R_A, R_B, R_S)\), and equivalences of descent data \((e^P_A, e^P_B, e^P_S) : (P \comp \inl, P \comp \inr, (\glue \tr)) \equivDD (P_A, P_B, P_S)\), and \((e^R_A, e^R_B, e^R_S) : (R \comp \inl, R \comp \inr, (\glue \tr)) \equivDD (R_A, R_B, R_S)\). Then there is an equivalence between the type of fiberwise equivalences \((x : \PO{S}{A}{B}) \to P(x) \simeq R(x)\) and equivalences of descent data \((P_A, P_B, P_S) \equivDD (R_A, R_B, R_S)\).

Additionally, for a fiberwise equivalence \(h\{x : \PO{S}{A}{B}\} : P(x) \simeq R(x)\) and its image \((h_A, h_B, h_S)\) under this map, we have the computation rules \(h_A(e^P_A(x)) = e^R_A(h(x))\) for all \(x : P(\inl(a))\), and \(h_B(e^P_B(x)) = e^R_B(h(x))\) for all \(x : P(\inr(b))\).
\label{lem:equiv-equiv-descent-data}
\end{lem}

\begin{proof}
We begin by describing the descent data corresponding to the type family \(x \mapsto (P(x) \simeq R(x))\). We claim that it is the descent data consisting of type families \(a \mapsto P_A(a) \simeq R_A(a)\) and \(b \mapsto P_B(b) \simeq R_B(b)\), together with the equivalence \(h \mapsto R_S \comp h \comp P_S^{-1}\) of type \(\{s : S\} \to (P_A(f s) \simeq R_A(f s)) \simeq (P_B(g s) \simeq R_B(g s))\). The claim is verified by providing equivalences \(e_A\{a : A\} : (P(\inl(a)) \simeq R(\inl(a))) \simeq (P_A(a) \simeq R_A(a))\) and \(e_B\{b : B\} : (P(\inr(b)) \simeq R(\inr(b))) \simeq (P_B(b) \simeq R_B(b))\), together with a coherence \(e_S\{s : S\}\) filling the square
\begin{center}
\begin{tikzcd}
  (P(\inl(f s)) \simeq R(\inl(f s)))
  \arrow[r, "e_A"]
  \arrow[d, "\glue(s) \tr"']
  & (P_A(f s) \simeq R_A(g s))
  \arrow[d, "R_S \comp \blank \comp P_S^{-1}"]
  \\
  (P(\inr(g s)) \simeq R(\inr(g s)))
  \arrow[r, "e_B"']
  & (P_B(g s) \simeq R_B(g s)).
\end{tikzcd}
\end{center}

The equivalences are defined as \(e_A(h) \defeq e^R_A \comp h \comp (e^P_A)^{-1}\) and \(e_B(h) \defeq e^R_B \comp h \comp (e^P_B)^{-1}\). Filling the square requires us to use function extensionality and characterization of transports in the family of equivalence types à la \cite[Equation 2.9.4]{UF13}. Applying the two, we are asked to find a homotopy
\begin{center}
\begin{tikzcd}[column sep=0pt]
  &[1em] P_A(f s)
  \arrow[dr, "(e^P_A)^{-1}"]
  &
  &[2em]
  & R_A(f s)
  \arrow[dr, "R_S"]
  &[0.5em]
  \\
  P_B(g s)
  \arrow[ur, "P_S^{-1}"]
  \arrow[dr, "(e^P_B)^{-1}"']
  &
  & P(\inl(f s))
  \arrow[r, "h"]
  & R(\inl(f s))
  \arrow[ur, "e^R_A"]
  \arrow[dr, "\glue(s) \tr"']
  &
  & R_B(g s)
  \\
  & P(\inr(g s))
  \arrow[ur, "(\glue(s) \tr)^{-1}"']
  &
  &
  & R(\inr(g s))
  \arrow[ur, "e^R_B"']
\end{tikzcd}
\end{center}
for every fiberwise equivalence \(h\{s : S\} : P(\inl(f s)) \simeq R(\inl(f s))\). The right square is exactly \(e^R_S\), while the left square is \(e^P_S\) with the equivalences along its boundary inverted.

It follows from \autoref{lem:sect-sect-descent-data} that the type of fiberwise equivalences \((x : \PO{S}{A}{B}) \to P(x) \simeq R(x)\) is equivalent to sections of the above descent data. We finish the proof by showing that this type of sections is equivalent to equivalences of descent data. Given such a section \((h_A, h_B, h_S)\), notice that \(h_A\) and \(h_B\) already have the expected type for them to be used in an equivalence of descent data. What is left is giving an equivalence between the type of coherences \(R_S \comp h_A(f s) \comp P_S^{-1} = h_B(g s)\) of the section and coherences \(h_B(g s) \comp P_S \htpy R_S \comp h_A(f s)\) of the equivalence, for all \(s : S\). This equivalence is obtained from function extensionality and transposition of \(P_S^{-1}\) to the other side of the homotopy.

To prove the computation rules, calculate that the first projection of the image of a fiberwise equivalence \(h\{x : \PO{S}{A}{B}\} : P(x) \simeq R(x)\) is \(h_A \judeq e_A^R \comp h \comp (e^P_A)^{-1}\), so the desired homotopy \(h_A \comp e^P_A \htpy e^R_A \comp h\) is constructed by canceling out the pair \((e^P_A)^{-1} \comp e^P_A\) on the left-hand side. Similarly for the second computation rule.
\end{proof}

By leveraging descent we are able to factor out proofs by a combination of pushout induction and univalence. This is useful for developments without computational HITs and univalence, since we can rebuild constructions over type families to be based on descent data instead, prove the two to be equivalent up to gluing the descent data into a type family, and then stay in the more computational ``local'' setting. We have already shown that type families correspond to descent data and fiberwise equivalences correspond to equivalences of descent data, and we proceed to introduce identity systems based on descent data.
\subsection{Identity systems}
\label{sec:org97e24b4}

In order to show that the zigzag construction is equivalent to the path spaces of pushouts, we introduce identity systems of pushouts. This is a restatement of the induction principle for pushout equality by Kraus and von Raumer \cite{KvR19}, defined to mirror the definition of identity systems for types, \eg \cite[Definition 11.2.1]{Rij22}. We recall the definition of identity systems.

\begin{defi}
A type family \(P : X \to \UU\) over a pointed type \((X, x_0)\) is an \textbf{identity system}~\href{http://archive.vojtechstep.eu/zigzag-construction/foundation.identity-systems.html\#the-predicate-of-being-an-identity-system}{\ExternalLink} at \(p_0 : P(x_0)\) if for all type families \(Q : (\Sigma (x : X). P(x)) \to \UU\), the evaluation map \({\evReflIdSystemTy(h) \defeq h(x_0, p_0) : ((u : \Sigma(x : X). P(x)) \to Q(u)) \to Q(x_0, p_0)}\)\label{ev-refl-id-system} has a section, in the sense of a converse map \(\indidsystemDD\) such that \(\evReflIdSystemTy \comp \indidsystemDD \htpy \id\).
\label{defn:identity-system}
\end{defi}

To express identity systems of descent data, take \(X\) to be a pushout. Then we may replace \(P\) with descent data, and \(x_0\) and \(p_0\) with points in a chosen component, \eg \(a_0 : A\) and \(p_0 : P_A(a_0)\). To translate \(Q\) and its sections, we make use of the flattening lemma, which states that the \(\Sigma\) type over the pushout is itself a pushout.

\begin{lemC}[{\cite[Proposition 1.9.2]{Bru16}} (Flattening)]
\href{https://archive.vojtechstep.eu/zigzag-construction/synthetic-homotopy-theory.flattening-lemma-pushouts.html\#proof-of-the-descent-data-statement-of-the-flattening-lemma}{\ExternalLink}
Given descent data \((P_A, P_B, P_S)\), a type family \(P\) over the pushout \(\PO{S}{A}{B}\), and an equivalence of descent data \((e_A, e_B, e_S)\) between \((P_A, P_B, P_S)\) and \((P \comp \inl, P \comp \inr, (\glue \tr))\), we have that the pushout of the total span diagram of \((P_A, P_B, P_S)\) is the cocone \((\Sigma(\inl). e_A, \Sigma(\inr). e_B, \Sigma(\glue). e_S^{-1})\). In particular, the target type is \(\Sigma(x : \PO{S}{A}{B}). P(x)\).
\label{lem:flattening}
\end{lemC}

It follows that an appropriate analogue of \(Q\) is descent data over the total span diagram. To differentiate between descent data over the base span diagram and descent data over the total one, we put a \(\Sigma\) in the subscripts of the latter, \ie the components are called \(Q_{\Sigma A}\), \(Q_{\Sigma B}\), and \(Q_{\Sigma S}\). This is a purely notational device.

\begin{defi}
Descent data \((P_A, P_B, P_S)\) over a span with a point \(a_0 : A\) is an \textbf{identity system} \href{https://archive.vojtechstep.eu/zigzag-construction/synthetic-homotopy-theory.identity-systems-descent-data-pushouts.html\#the-predicate-of-being-an-identity-system-on-descent-data-for-pushouts}{\ExternalLink} at \(p_0 : P_A(a_0)\) if for all descent data \((Q_{\Sigma A}, Q_{\Sigma B}, Q_{\Sigma S})\) over its total span, the evaluation map \(\evReflIdSystemDD(t_A, t_B, t_S) \defeq t_A(a_0, p_0) : \sectPO(Q_{\Sigma A}, Q_{\Sigma B}, Q_{\Sigma B}) \to Q_{\Sigma A}(a_0, p_0)\)\mlabel{ev-refl-id-system-DD}
has a section.
\end{defi}

Just like the based identity types are a canonical example of identity systems, we have a canonical identity system of descent data.

\begin{constr}
\href{https://archive.vojtechstep.eu/zigzag-construction/synthetic-homotopy-theory.identity-systems-descent-data-pushouts.html\#the-canonical-descent-data-for-families-of-identity-types-is-an-identity-system}{\ExternalLink}
For a point \(a_0 : A\), define the descent data \((I_A, I_B, I_S)\) by posing \(I_A(a) \defeq (\inl(a_0) = \inl(a))\), \(I_B(b) \defeq (\inl(a_0) = \inr(b))\), and \(I_S\{s\}(p) \defeq p \concat \glue(s)\).

By computation of transports in based identity types \cite[Lemma 2.11.2]{UF13}, this descent data is equivalent to the descent data induced by the type family \(I(x) \defeq (\inl(a_0) = x)\).
\end{constr}

The formalization shows that a pointed type family is an identity system if and only if the induced pointed descent data is an identity system, but for this paper we limit ourselves to the following theorem, which combines a proof of \((I_A, I_B, I_S)\) being an identity system with unique uniqueness of initial objects:

\begin{thm}
\href{https://archive.vojtechstep.eu/zigzag-construction/synthetic-homotopy-theory.identity-systems-descent-data-pushouts.html\#unique-uniqueness-of-identity-systems}{\ExternalLink}
Consider a span diagram with a point \(a_0 : A\). For any identity system \((P_A, P_B, P_S)\) at \(p_0 : P_A(a_0)\), there is a unique triple \((e_A, e_B, e_S)\) consisting of
\begin{alignat*}{2}
  &e_A \{a : A\}: (\inl(a_0) = \inl(a)) \simeq P_A(a)  &\qquad{}e_B \{b : B\}&: (\inl(a_0) = \inr(b)) \simeq P_B(b) \\
  &\multispan{3}{\hfill$e_S \{s : S\}: (p : \inl(a_0) = \inl(fs)) \to e_B(p \concat (Hs)) = P_S(e_A(p))$\hfill}
\end{alignat*}
such that \(e_A(\refl) = p_0\).
\label{thm:unique-uniqueness-id-system}
\end{thm}

\begin{proof}
By descent, there is a unique type family \(P\) over \(\PO{S}{A}{B}\) equipped with an equivalence of descent data \((d_A, d_B, d_S)\) between \((P \comp \inl, P \comp \inr, (\glue \tr))\) and \((P_A, P_B, P_S)\). The data we want to construct is an equivalence \((e_A, e_B, e_S)\) between \((I_A, I_B, I_S)\) and \((P_A, P_B, P_S)\) such that \(e_A\{a_0\}(\refl) = p_0\). By \autoref{lem:equiv-equiv-descent-data} the type of equivalences is equivalent to the type of fiberwise equivalences \(e\{x\} : (\inl(a_0) = x) \simeq P(x)\), and the side condition is equivalent to requiring \(e\{\inl(a_0)\}(\refl) = d_A^{-1}(p_0)\). By the fundamental theorem of identity types \cite[Theorem 11.2.2]{Rij22}, there is a unique such equivalence if and only if \(P\) is an identity system at \(d_A^{-1}(p_0)\). So consider a type family \(Q_{\Sigma}\) over \(\Sigma(x : \PO{S}{A}{B}). P(x)\) and let us find a section to the evaluation map \(\evReflIdSystemTy\). The type family induces the descent data \((Q_{\Sigma A}, Q_{\Sigma B}, Q_{\Sigma S})\), which is defined as
\begin{alignat*}{2}
  &Q_{\Sigma A}(a, p) \defeq Q_{\Sigma}(\inl(a), d_A^{-1}(p))  &\qquad{}Q_{\Sigma B}(b, p)&\defeq Q_{\Sigma}(\inr(b), d_B^{-1}(p)) \\
  &\multispan{3}{\hfill$Q_{\Sigma S}(s, p) \defeq (\glue(s), d'_S(p)) \tr : Q_{\Sigma A}(f s, p) \simeq Q_{\Sigma B}(g s, P_S(p))$\hfill}
\end{alignat*}
where \(d'_S\{s\} : (\glue(s) \tr) \comp d_A^{-1}\{f s\} \htpy d_B^{-1}\{g s\} \comp P_S\{s\}\) is obtained from the commuting square \(d_S\{s\}^{-1} : d_B\{g s\} \comp (\glue(s) \tr) \htpy P_S\{s\} \comp d_A\{f s\}\) by transposing the equivalences \(d_A\) and \(d_B\) along the homotopy.

Observe that \(Q_{\Sigma A}\) fits into the commuting diagram
\phantomsection
\label{eqn:id-system-square}
\begin{center}
\begin{tikzcd}
  ((u : \Sigma (x : \PO{S}{A}{B}). P(x)) \to Q_{\Sigma} (u))
  \arrow[r, outer sep=3pt, "\evReflIdSystemTy"]
  \arrow[d, "\depCoconeMapPO"', "\simeq"]
  & Q_{\Sigma}(\inl(a_0), d^{-1}(p_0))
  \arrow[d, "\id", "\simeq"'] \\
  \sectPO(Q_{\Sigma A}, Q_{\Sigma B}, Q_{\Sigma S})
  \arrow[r, "\evReflIdSystemDD"']
  & Q_{\Sigma A}(a_0, p_0),
\end{tikzcd}
\end{center}
where the left map is an equivalence by \autoref{lem:flattening}. By the assumption that \((P_A, P_B, P_S)\) is an identity system at \(p_0\), we get that the bottom map has a section. We conclude that the top map is a conjugation of the bottom map by equivalences, so it has a section as well.
\end{proof}

We are going to use \autoref{thm:unique-uniqueness-id-system} at the end of \autoref{sec:idpo} to obtain the equivalences between path spaces and the zigzag construction.
\section{Sequential colimits}
\label{sec:seqcol}
In this section, we treat sequential colimits in homotopy type theory, following Sojakova, van Doorn and Rijke \cite{SvDR20}. We use different names for core concepts to stay consistent with the rest of the agda-unimath library --- mainly ``sequences'', ``natural transformations of sequences'' and ``fibered sequences'' are renamed to ``sequential diagrams'', ``morphisms of sequential diagrams'' and ``dependent sequential diagrams''. Apart from the necessary definitions to define the zigzag construction, the main goal is to prove \autoref{lem:preserves-cubes}, which states that a sequence of cubes of sections induces a square of sections in the colimit.

\begin{defi}
A \textbf{sequential diagram} \href{https://archive.vojtechstep.eu/zigzag-construction/synthetic-homotopy-theory.sequential-diagrams.html\#definition}{\ExternalLink} is a pair \((A, a)\), consisting of a family of types \(A : \N \to \UU\), and a connecting family of maps \(a_n : A_n \to A_{n+1}\). When the maps are clear from the context, we use \(A_{\bullet}\) for the sequential diagram.

Consider a sequential diagram \((A, a)\).
\begin{enumerate}[i.]
\item A \textbf{cocone} \href{https://archive.vojtechstep.eu/zigzag-construction/synthetic-homotopy-theory.cocones-under-sequential-diagrams.html\#cocones-under-sequential-diagrams-1}{\ExternalLink} to a type \(X\) is a pair \((i, H)\), which consists of a family of maps \(i_n : A_n \to X\) and a family of homotopies \(H_n : i_n \htpy i_{n+1} \comp a_n\).
\item A \textbf{dependent cocone} \href{https://archive.vojtechstep.eu/zigzag-construction/synthetic-homotopy-theory.dependent-cocones-under-sequential-diagrams.html\#dependent-cocones-under-sequential-diagrams-1}{\ExternalLink} to a type family \(P : X \to \UU\) over a cocone \((i, H)\) to \(X\) is a pair \((i', H')\), consisting of a family of dependent maps \(i'_n : (a : A_n) \to P(i_n(a))\), and a family of dependent homotopies \(H'_n : (H_n \tr) \comp i'_n \htpy i'_{n+1} \comp a_n\).
\item The \textbf{cocone map} \href{https://archive.vojtechstep.eu/zigzag-construction/synthetic-homotopy-theory.cocones-under-sequential-diagrams.html\#postcomposing-cocones-under-a-sequential-diagram-with-a-map}{\ExternalLink} accepts a cocone \((i, H)\) to \(X\) and a function \(h : X \to Y\), and constructs the cocone \((n \mapsto h \comp i_n, n\,a \mapsto \ap_h(H_n(a)))\).
\item The \textbf{dependent cocone map} \href{https://archive.vojtechstep.eu/zigzag-construction/synthetic-homotopy-theory.dependent-cocones-under-sequential-diagrams.html\#obtaining-dependent-cocones-under-sequential-diagrams-by-postcomposing-cocones-under-sequential-diagrams-with-dependent-maps}{\ExternalLink}, written \(\depCoconeMapSC\)\mlabel{dep-cocone-map-sc}, accepts a cocone \((i, H)\) to \(X\) and a dependent map \(s : (x : X) \to P(x)\), and constructs the dependent cocone \((n \mapsto s \comp i_n, n \, a \mapsto \apd_s(H_n(a)))\).
\item A cocone \(c\) to \(X\) is a \textbf{sequential colimit} if either its cocone map is an equivalence for all \(Y\) (the \textbf{universal property} \href{https://archive.vojtechstep.eu/zigzag-construction/synthetic-homotopy-theory.universal-property-sequential-colimits.html\#the-universal-property-of-sequential-colimits-1}{\ExternalLink}), or its dependent cocone map is an equivalence for all \(P\) (the \textbf{dependent universal property} \href{https://archive.vojtechstep.eu/zigzag-construction/synthetic-homotopy-theory.dependent-universal-property-sequential-colimits.html\#the-dependent-universal-property-of-sequential-colimits-1}{\ExternalLink}). By convention, we call the target type \(A_{\infty}\), the \(i_n\) maps \(\inn_n\)\mlabel{inn} (``point constructors''), and the \(H_n\) homotopies \(\kapn_n\)\mlabel{kapn} (``path constructors''). If we talk about multiple colimits, we distinguish their constructors by an additional index, \eg \(\inn_A^n\) and \(\inn_B^n\). We refer to both \(X\) and the cocone as ``sequential colimit''.
\item The type of \textbf{dependent sequential diagrams} \href{https://archive.vojtechstep.eu/zigzag-construction/synthetic-homotopy-theory.dependent-sequential-diagrams.html\#dependent-sequential-diagrams-1}{\ExternalLink} over \((A, a)\) is the type of pairs \((P, p)\), with \(P_n : A_n \to \UU\) a family of type families, and \(p_n\{a : A_n\} : P_n(a) \to P_{n+1}(a_n(a))\) a family of fiberwise maps. When the maps are clear from the context, we denote the dependent sequential diagram \(P_{\bullet}\).
\item The type of \textbf{sections} \href{https://archive.vojtechstep.eu/zigzag-construction/synthetic-homotopy-theory.dependent-sequential-diagrams.html\#sections-of-dependent-sequential-diagrams}{\ExternalLink} of a dependent sequential diagram \((P, p)\) is the type of pairs \((s, K)\), consisting of a family of dependent functions \(s_n : (a : A_n) \to P_n(a)\) and a family of squares of sections \(K_n\), visualized in \autoref{subfig:coherence-sect}.
\item The identity type of sections \((t, L) = (s, K)\) is characterized by the type of \textbf{homotopies}~\href{https://archive.vojtechstep.eu/zigzag-construction/synthetic-homotopy-theory.dependent-functoriality-sequential-colimits.html\#characterization-of-identity-types-of-sections-of-descent-data}{\ExternalLink} of sections, which are families of homotopies \(F_n : t_n \htpy s_n\), equipped with a family of coherences described in \autoref{subfig:coherence-htpy}.
\item A \textbf{morphism}~\href{https://archive.vojtechstep.eu/zigzag-construction/synthetic-homotopy-theory.morphisms-sequential-diagrams.html\#morphisms-of-sequential-diagrams-1}{\ExternalLink} from \((A, a)\) to \((B, b)\) is a pair \((f, H)\), where \(f_n : A_n \to B_n\) is a family of functions, and \(H_n : b_n \comp f_n \htpy f_{n+1} \comp a_n\) is a family of commuting squares. We can construct the identity morphisms~\href{https://archive.vojtechstep.eu/zigzag-construction/synthetic-homotopy-theory.morphisms-sequential-diagrams.html\#the-identity-morphism-of-sequential-diagrams}{\ExternalLink} \((\id, \reflhtpy)\), and morphisms compose~\href{https://archive.vojtechstep.eu/zigzag-construction/synthetic-homotopy-theory.morphisms-sequential-diagrams.html\#composition-of-morphisms-of-sequential-diagrams}{\ExternalLink} in the expected way.
\item For a dependent sequential diagram \((P, p)\) over \((A, a)\) and a dependent sequential diagram \((Q, q)\) over \((B, b)\), the type of \textbf{fiberwise morphisms} \href{https://archive.vojtechstep.eu/zigzag-construction/synthetic-homotopy-theory.dependent-functoriality-sequential-colimits.html\#fiberwise-morphisms-of-dependent-diagrams}{\ExternalLink} from \(P_{\bullet}\) to \(Q_{\bullet}\) over a morphism \((f, H) : A_{\bullet} \to B_{\bullet}\) is the type of pairs \((g, G)\), where \(g_n\{a : A_n\} : P_n(a) \to Q_n(f_n(a))\) is a family of fiberwise maps, and \(G_n\{a:A_n\} : (H_n(a) \tr) \comp q_n\{f_n(a)\} \comp g_n\{a\} \htpy g_{n+1}\{a_n(a)\} \comp p_n\{a\}\) is a family of dependent squares over \(H_n\), visualized in \autoref{subfig:coherence-fibhom}.
\item The \textbf{shift} \href{https://archive.vojtechstep.eu/zigzag-construction/synthetic-homotopy-theory.shifts-sequential-diagrams.html\#shifts-of-sequential-diagrams-1}{\ExternalLink} \((A_{+1}, a_{+1})\) is the sequential diagram obtained from \((A, a)\) by forgetting the first type and map, \ie \((A_{+1})(n) \defeq A_{n+1}\) and \((a_{+1})(n) \defeq a_{n+1}\). We write \(A_{\bullet+1}\) for the shift of \(A_{\bullet}\).
\end{enumerate}
\end{defi}

\begin{figure}
\begin{subfigure}{.25\textwidth}
\begin{tikzcd}[cramped]
  P_n \arrow[r, "p_n"] & P_{n+1} \\
  A_n \arrow[u, "s_n" pos=0.4] \arrow[r, "a_n"'] \arrow[ru, phantom, "K_n"] & A_{n+1} \arrow[u, "s_{n+1}"' pos=0.4]
\end{tikzcd}
\centering\caption{}
\label{subfig:coherence-sect}
\end{subfigure}
\begin{subfigure}{.39\textwidth}
\begin{tikzcd}[column sep=60pt, row sep=25pt, cramped]
  P_n \arrow[r, "p_n"]
  \arrow[r, phantom, "L_n" pos=0.3, yshift=-12pt] & P_{n+1} \\
  A_n \arrow[r, "a_n"'] \arrow[u, "t_n", bend left, ""{name=S0, anchor=center}] \arrow[u, "s_n"', bend right, ""{name=T0, anchor=center}]
  \arrow[r, phantom, "K_n" pos=0.65, yshift=12pt] & A_{n+1} \arrow[u, "t_{n+1}", bend left=45, ""{name=S1, anchor=center}] \arrow[u, "s_{n+1}"', bend right=45, ""{name=T1, anchor=center}]
  \arrow[from=S0, to=T0, phantom, "F_n"]
  \arrow[from=S1, to=T1, phantom, "F_{n+1}"]
\end{tikzcd}
\centering\caption{}
\label{subfig:coherence-htpy}
\end{subfigure}
\hspace{-1em}
\begin{subfigure}{.35\textwidth}
\begin{tikzcd}[sep=small, cramped]
  P_n \arrow[dd, two heads]
  \arrow[dr, "g_n"'] \arrow[rr, crossing over, "p_n"]
  & & P_{n+1} \arrow[dd, two heads]
  \arrow[dr, "g_{n+1}"] \\
  & Q_n
  \arrow[rr, "/" {marking, near end}, crossing over, "q_n"' pos=0.3]
  \arrow[ur, phantom, "G_n"]
  & & Q_{n+1} \\
  A_n\arrow[dr, "f_n"'] \arrow[rr, "a_n", pos=0.7]
  & & A_{n+1} \arrow[dr, "f_{n+1}" pos=0.3]
  \\
  & B_n \arrow[from=uu, two heads, crossing over] \arrow[rr, "b_n"']
  \arrow[ur, phantom, "H_n"]
  & & B_{n+1} \arrow[from=uu, two heads, crossing over]
\end{tikzcd}
\centering\caption{}
\label{subfig:coherence-fibhom}
\end{subfigure}
\caption{Coherence diagrams of \subref{subfig:coherence-sect} sections, \subref{subfig:coherence-htpy} homotopies of sections, and \subref{subfig:coherence-fibhom} fiberwise morphisms. Note that dependent functions point up. Figure \subref{subfig:coherence-htpy} should be read as a cylinder, with front square \(K_n\) and back square \(L_n\). In figure \subref{subfig:coherence-fibhom}, the slash on the arrow \(q_n\) indicates that there is an implicit transport along \(H_n\), which places the dependent square over the bottom square.}
\end{figure}

The rest of this sections develops vocabulary and lemmas used to prove  \autoref{lem:preserves-cubes}.

\begin{lem}
\href{https://archive.vojtechstep.eu/zigzag-construction/synthetic-homotopy-theory.dependent-functoriality-sequential-colimits.html\#a-sequence-of-cubes-of-sections-induces-a-square-of-sections-in-the-colimit}{\ExternalLink}
Consider a morphism of sequential diagrams \((f, H) : A_{\bullet} \to B_{\bullet}\), two type families \(P\) and \(Q\) over \(A_{\infty}\) and \(B_{\infty}\) in the same universe, sections \(t_{\bullet}\) and \(s_{\bullet}\) of \(P_{\bullet}\) and \(Q_{\bullet}\), respectively, and a fiberwise equivalence \(e\{a\} : P(a) \simeq Q(f_{\infty}(a))\).
Then given a family of homotopies \(F_n : e_n \comp t_n \htpy s_n \comp f_n\), and a family of cubes as in \autoref{fig:lemma-fig}, we get a homotopy \(e \comp t_{\infty} \htpy s_{\infty} \comp f_{\infty}\).
\label{lem:preserves-cubes}
\end{lem}

\begin{figure}
\begin{subfigure}{.7\textwidth}
\begin{tikzcd}[sep=small]
  (P \comp \inn^n_A)
  \arrow[dr, "(C_n \tr)\, \comp\, e" pos=0.8] \arrow[rr, crossing over, "\kapn^n_A \tr"]
  & & (P \comp \inn^{n+1}_A) \arrow[from=dd, "t_{n+1}"' pos=0.8]
  \arrow[dr, "(C_{n+1} \tr)\, \comp\, e"]
  \\
  & (Q \comp \inn^n_B)
  \arrow[rr, "/" {marking, near end}, crossing over, "\kapn^n_B \tr"' pos=0.2]
  & & (Q \comp \inn^{n+1}_B)
  \\
  A_n\arrow[dr, "f_n"'] \arrow[rr, "a_n", pos=0.8] \arrow[uu, "t_n"] \arrow[ur, phantom, "F_n"]
  & & A_{n+1} \arrow[dr, "f_{n+1}"' {pos=0.3, outer sep=-2pt}] \arrow[ur, phantom, "F_{n+1}"]
  \\
  & B_n \arrow[to=uu, crossing over, "s_n" pos=0.3] \arrow[rr, "b_n"']
  \arrow[ur, phantom, "H_n"]
  & & B_{n+1} \arrow[uu, crossing over, "s_{n+1}"']
\end{tikzcd}
\centering\caption{}
\label{subfig:lemma-input}
\end{subfigure}
\begin{subfigure}{.29\textwidth}
\begin{tikzcd}[sep=small]
P \arrow[dr, "e"] \\[1ex]
&[2.7em] Q \\
A_{\infty} \arrow[uu, "t_{\infty}"] \arrow[dr, "f_{\infty}"'] \\
& B_{\infty} \arrow[uu, "s_{\infty}"']
\end{tikzcd}
\centering\caption{}
\label{subfig:lemma-output}
\end{subfigure}
\caption{\label{fig:lemma-fig}\subref{subfig:lemma-input} Dependent cubes of sections which induce \subref{subfig:lemma-output} a dependent square of sections in the colimit. The top square of \subref{subfig:lemma-input} is described in \autoref{defn:squares-over}.}
\end{figure}

The result by itself is intuitive enough, and one might not even write it down when working on paper: it shows a limited version of a homotopy of sections \((e_{\bullet} \comp t_{\bullet}) \htpy (s_{\bullet} \comp f_{\bullet})\) inducing a homotopy \((e_{\infty} \comp t_{\infty}) \htpy (s_{\infty} \comp f_{\infty})\). We note, however, that it is not a straightforward application of functoriality: the type of the input homotopy mixes morphisms, fiberwise morphisms and sections, so the two composition operators are different, and neither one is the composition of simple morphisms. We choose to tailor the statements of the lemmas to the proof of correctness of the zigzag construction, imposing stronger assumptions than necessary if it simplifies the proofs, see \eg \autoref{remark:equiv-good}.

We begin by introducing constructions mediating between diagrams and colimits.

\begin{constr}
\href{https://archive.vojtechstep.eu/zigzag-construction/synthetic-homotopy-theory.functoriality-sequential-colimits.html\#a-morphism-of-sequential-diagrams-induces-a-map-of-sequential-colimits}{\ExternalLink}
A morphism of sequential diagrams \(f_{\bullet} : A_{\bullet} \to B_{\bullet}\) induces a map of colimits \(f_{\infty} : A_{\infty} \to B_{\infty}\), by applying the universal property to the cocone constructed by precomposing the cocone \(B_{\infty}\) with \(f_{\bullet}\). Since the cocone map is an equivalence, \(f_{\infty}\) is the unique such map equipped with a family of computation rules \(C_n : f_{\infty} \comp \inn_A^n \htpy \inn_B^n \comp f_n\), which are equipped with a coherence as illustrated by the prism in \autoref{subfig:prism-coh}.
\label{lem:functoriality}
\end{constr}

\begin{figure}
\begin{subfigure}{.34\textwidth}
\begin{tikzcd}[row sep=small]
  A_n
  \arrow[dd, "f_n"']
  \arrow[rr, "a_n"]
  \arrow[dr, "\inn_A^n"', near end]
  && A_{n+1}
  \arrow[dd, "f_{n+1}"]
  \arrow[dl, "\inn_A^{n + 1}", near end] \\
  & A_{\infty} \\
  B_n
  \arrow[rr, "b_n", near start]
  \arrow[dr, "\inn_B^n"']
  && B_{n+1}
  \arrow[dl, "\inn_B^{n + 1}"] \\
  & B_{\infty}
  \arrow[from=uu, "f_{\infty}", near start, crossing over]
\end{tikzcd}
\vspace{.35ex}
\centering\caption{}
\label{subfig:prism-coh}
\end{subfigure}
\begin{subfigure}{.65\textwidth}
\begin{tikzcd}[column sep=small, row sep=small]
  P_n \arrow[dddr, two heads, bend right=15] \arrow[rr, "\kapn_A^n\tr"] \arrow[dr, "e"]
  & & P_{n+1} \arrow[dddr, two heads, bend right=15] \arrow[dr, "e"] \\
  & (Q \comp f_{\infty})_n \arrow[dd, two heads]
  \arrow[dr, "C_n\tr"'] \arrow[rr, crossing over, near start, "\kapn_A^n\tr"]
  & & (Q \comp f_{\infty})_{n+1} \arrow[dd, two heads]
  \arrow[dr, "C_{n+1}\tr"] \\
  & & Q_n
  \arrow[rr, "/" {marking, near end}, crossing over, "\kapn_B^n\tr" pos=0.3]
  & & Q_{n+1} \\
  & A_n\arrow[dr, "f_n"'] \arrow[rr, "a_n", pos=0.7]
  & & A_{n+1} \arrow[dr, "f_{n+1}"]
  \\
  & & B_n \arrow[from=uu, two heads, crossing over] \arrow[rr, "b_n"']
  \arrow[ur, phantom, "H_n"]
  & & B_{n+1} \arrow[from=uu, two heads, crossing over]
\end{tikzcd}
\centering\caption{}
\label{fig:fibhom-coh}
\end{subfigure}
\caption{Coherences of \subref{subfig:prism-coh} the computation rules of \(f_{\infty}\), and \subref{fig:fibhom-coh} the fiberwise morphism induced by \(e\{a : A_{\infty}\} : P(a) \simeq Q(f_{\infty}(a))\).}
\end{figure}

\begin{constr}
\href{https://archive.vojtechstep.eu/zigzag-construction/synthetic-homotopy-theory.descent-data-sequential-colimits.html\#descent-data-induced-by-families-over-cocones-under-sequential-diagrams}{\ExternalLink}
A type family \(P\) over \(A_{\infty}\) induces a dependent diagram \(P_{\bullet}\), consisting of type families \(P_n(a) \defeq P(\inn_n(a))\) and maps \(\kapn_n \tr : (a : A_n) \to P(\inn_n(a)) \to P(\inn_{n+1}(a_n(a)))\).
\end{constr}

Observe that the type of dependent cocones with vertex \(P\) is exactly the same as the type of sections of the induced type family \(P_{\bullet}\).

\begin{constr}
Given a dependent map \(s : (a : A_{\infty}) \to P(a)\), the \textbf{induced section}~\href{https://archive.vojtechstep.eu/zigzag-construction/synthetic-homotopy-theory.dependent-functoriality-sequential-colimits.html\#sections-of-descent-data-induced-by-dependent-functions}{\ExternalLink} of the dependent diagram \(P_{\bullet}\) is the section \(s_{\bullet} \defeq \depCoconeMapSC(s)\).
\end{constr}

\begin{constr}
\href{https://archive.vojtechstep.eu/zigzag-construction/synthetic-homotopy-theory.dependent-functoriality-sequential-colimits.html\#the-dependent-map-induced-by-a-section-of-descent-data}{\ExternalLink}
Given a type family \(P\) over \(A_{\infty}\) and a section \(s_{\bullet}\) of \(P_{\bullet}\), the dependent map \(s_{\infty}(a : A_{\infty}) : P(a)\) is defined by \(s_{\bullet}\) viewed as a dependent cocone.
\end{constr}

\begin{constr}
\href{https://archive.vojtechstep.eu/zigzag-construction/synthetic-homotopy-theory.dependent-functoriality-sequential-colimits.html\#fiberwise-morphism-of-dependent-sequential-diagrams-induced-by-a-fiberwise-equivalence}{\ExternalLink}
Given a morphism of sequential diagrams \(f_{\bullet} : A_{\bullet} \to B_{\bullet}\), a fiberwise map \({e\{a : A_{\infty}\}: P(a) \to Q(f_{\infty}(a))}\) induces a fiberwise morphism \(e_{\bullet}\) from \(P_{\bullet}\) to \(Q_{\bullet}\) over \(f_{\bullet}\). The maps \(e_n\{a : A_n\}\) are defined as
\begin{tikzcd}[cramped, sep=small]
P_n(a) \arrow[r, "e" pos=0.4] &
Q(f_{\infty}\inn_A^n(a)) \arrow[r, "C_n \tr" pos=0.4] &[10pt]
Q_n(f_n(a)) ,
\end{tikzcd}
where \(C_n\{a : A_n\} : f_{\infty}\inn_A^n(a) = \inn_B^nf_n(a)\) is the computation rule of \(f_{\infty}\).

To define the coherences, we help ourselves with the diagram in \autoref{fig:fibhom-coh}. The desired dependent squares are defined by pasting --- the top square on the left is non-dependent, and is filled by transports commuting with fiberwise maps \cite[Lemma 2.3.11]{UF13}. To fill the right square, we use distributivity of transport over path concatenation \cite[Lemma 2.3.9]{UF13} and left whiskering \cite[Lemma 2.3.10]{UF13} to adjust the boundary so that we are asked to fill a homotopy of two transports in the same family, over different paths. The two paths are then shown to be equal by the coherence of \(C_n\), making the two transports homotopic.
\label{defn:squares-over}
\end{constr}

\begin{constr}
\href{https://archive.vojtechstep.eu/zigzag-construction/synthetic-homotopy-theory.dependent-functoriality-sequential-colimits.html\#composition-of-a-section-and-morphism}{\ExternalLink}
Given a morphism of sequential diagrams \((f, H) : A_{\bullet} \to B_{\bullet}\), a type family \(Q : B_{\infty} \to \UU\) and a section \((s, K)\) of \(Q_{\bullet}\), we construct the section \(s_{\bullet} \compr f_{\bullet}\)\label{compr} of the dependent sequential diagram over \(A_{\bullet}\) induced by \((Q \comp f_{\infty}) : A_{\infty} \to \UU\). The symbol ``\(\compr\)'' serves as a reminder of the shape of this composition. The maps \((s_{\bullet} \compr f_{\bullet})_n\) are given by the composites \((C_n^{-1} \tr) \comp s_n \comp f_n : (a : A_n) \to Q(f_{\infty}(\inn_A^n(a)))\).

The intuition for constructing the coherences is visualized in \autoref{subfig:sect-hom-coh}. It consists of pasting the squares \(H_n\), \(K_n\), and the dependent square from \autoref{defn:squares-over}. The top square is ``flipped'' from front to back as a non-dependent square, meaning that the transport stays on the same arrow. The flipping is defined by transposing equivalences along homotopies, to go from \((H_n \tr) \comp (\kapn_B^n \tr) \comp (C_n \tr) \htpy (C_{n+1} \tr) \comp (\kapn_A^n \tr)\) to \((C_{n+1}^{-1} \tr) \comp (H_n \tr) \comp (\kapn_B^n \tr) \htpy (\kapn_A^n \tr) \comp (C_n^{-1} \tr)\). In the present graphical framework we cannot formally paste \(H_n\) and \(K_n\), so this diagram is mainly for illustrative purposes.
\label{defn:sect-precomp}
\end{constr}

\begin{figure}
\begin{tikzcd}[column sep=small, row sep=small]
  (Q \comp f_{\infty})_n
  \arrow[from=dr, "C_n^{-1}\tr"] \arrow[rr, crossing over, "\kapn_A^n\tr"]
  & & (Q \comp f_{\infty})_{n+1}
  \arrow[from=dr, "C_{n+1}^{-1}\tr"'] \\
  & Q_n
  \arrow[rr, "/" {marking, near end}, crossing over, "\kapn_B^n\tr" pos=0.3]
  & & Q_{n+1} \\
  A_n\arrow[dr, "f_n"'] \arrow[rr, "a_n", pos=0.7]
  & & A_{n+1} \arrow[dr, "f_{n+1}"]
  \\
  & B_n \arrow[uu, crossing over, "s_n" near start] \arrow[rr, "b_n"']
  & & B_{n+1} \arrow[uu, crossing over, "s_{n+1}"']
\end{tikzcd}
\caption{\label{subfig:sect-hom-coh}Coherences of the section \(s_{\bullet} \compr f_{\bullet}\) of \((Q \comp f_{\infty})_{\bullet}\).}
\end{figure}

We prove a lemma which allows us to construct homotopies to \(s_{\bullet} \compr f_{\bullet}\) in a more straightforward way which does not involve flipping the top square.

\begin{figure}
\begin{tikzcd}[column sep=small, row sep=small]
  (Q \comp f_{\infty})_n \arrow[from=dd, "t_n"]
  \arrow[dr, "C_n\tr"'] \arrow[rr, crossing over, "\kapn_A^n\tr"]
  & & (Q \comp f_{\infty})_{n+1} \arrow[from=dd, "t_{n+1}"'{pos=.7}]
  \arrow[dr, "C_{n+1}\tr"] \\
  & Q_n
  \arrow[rr, "/" {marking, near end}, crossing over, "\kapn_B^n\tr" pos=0.25]
  & & Q_{n+1} \\
  A_n\arrow[dr, "f_n"'] \arrow[rr, "a_n", pos=0.7]
  & & A_{n+1} \arrow[dr, "f_{n+1}"]
  \\
  & B_n \arrow[uu, crossing over, "s_n"{pos=.25}] \arrow[rr, "b_n"']
  & & B_{n+1} \arrow[uu, crossing over, "s_{n+1}"']
\end{tikzcd}
\caption{\label{subfig:cube-coh}Coherences of squares of sections \(F_n\). The faces of the cube are \(H_n\), \(F_n\), \(F_{n+1}\), the coherences of \(t_{\bullet}\) and \(s_{\bullet}\), and the dependent square from \autoref{fig:fibhom-coh}.}
\end{figure}

\begin{lem}
\href{https://archive.vojtechstep.eu/zigzag-construction/synthetic-homotopy-theory.dependent-functoriality-sequential-colimits.html\#cubes-of-sections-of-pulled-back-descent-data-induce-homotopies-of-composed-sections}{\ExternalLink} Given a morphism \((f, H)\) and a section \((s, K)\) as in \autoref{lem:preserves-comp}, and additionally a section \((t, L)\) of \((Q \comp f_{\infty})_{\bullet}\), squares of sections \(F_n\{a\}\) witnessing paths \(C_n \tr (t_n(a)) = s_n(f_n(a))\), and coherences filling the cubes in \autoref{subfig:cube-coh}, then the sections \(t\) and \(s_{\bullet} \compr f_{\bullet}\) are homotopic.
\label{lem:cube-compr}
\end{lem}

\begin{proof}
To show that \(t\) and \(s_{\bullet} \compr f_{\bullet}\) are homotopic, we need to complete the partial cube in \autoref{subfig:sect-hom-coh}, extended with the dependent maps \(t_n\), \(t_{n+1}\) and the back face \(L_n\). The homotopy of maps \(t_n \htpy (C_n^{-1} \tr) \comp s_n \comp f_n\) is obtained from \(F_n\) by transposing \((C_n \tr)\). A technical proof by homotopy algebra shows that just as we can ``flip'' \((C_n \tr)\) and the top square, we can coherently flip the cube filler to get a cube with the expected top.
\end{proof}

\begin{lem}
\href{https://archive.vojtechstep.eu/zigzag-construction/synthetic-homotopy-theory.dependent-functoriality-sequential-colimits.html\#sequential-colimits-preserve-compositions-of-morphisms-and-sections}{\ExternalLink} The dependent map \((s_{\bullet} \compr f_{\bullet})_{\infty} : (a : A_{\infty}) \to Q(f_{\infty}(a))\) induced by \autoref{defn:sect-precomp} is homotopic to the composition \(s_{\infty} \comp f_{\infty}\).
\label{lem:preserves-comp}
\end{lem}

\begin{proof}
Since dependent functions out of \(A_{\infty}\) are fully determined by their induced dependent cocones, it suffices to show that the two induced sections of the dependent diagram \((Q \comp f_{\infty})_{\bullet}\), namely \((s_{\infty} \comp f_{\infty})_{\bullet}\) and \(((s_{\bullet} \compr f_{\bullet})_{\infty})_{\bullet}\), are homotopic. The latter is homotopic to its defining dependent cocone \((s_{\bullet} \compr f_{\bullet})\), so it suffices to show a homotopy between that and \((s_{\infty} \comp f_{\infty})_{\bullet}\).

By \autoref{lem:cube-compr}, it suffices to give homotopies between \((C_n \tr) \comp s_{\infty} \comp f_{\infty} \comp \inn_A^n\) and \(s_n \comp f_n\), and appropriate cube fillers.

First recall that \(s_{\infty}\) is defined by the universal property, so it comes equipped with computation rules and coherences. They are depicted in \autoref{subfig:prism-coh-dep}: the side squares are the computation rules, and the filler of the prism of dependent maps is the coherence.

Then to give the required homotopy of maps, construct the square of sections as shown in \autoref{subfig:htpy-sf} by taking \(a \mapsto \apd_{s_{\infty}}(C_n(a))\) for the left square, and the computation rule of \(s_{\infty}\) on \(\inn_B^n(f_n(a))\) for the right homotopy.

The coherences are obtained by constructing coherences for the left and right subshapes from \autoref{subfig:htpy-sf} separately and then pasting them. Coherence of the left square follows from coherence of \(C_n\), as the boundary involves the homotopies \(C_n\), \(C_{n+1}\), \(\kapn_A^n\), \(\kapn_B^n\) and \(H_n\). For some insight, recall that the top square was defined by deploying technical lemmas to manipulate the boundary to the form \(p \tr q = r \tr q\) for certain paths \(p\) and \(r\) being the two sides of the coherence of \(C_n\), and then using the coherence. Here we then go through these technical manipulations and show how they behave when composed with variations of \(\apd_{s_{\infty}}\). Since only the top face mentions the coherence of \(C_n\), it may not be apparent that it commutes with the other faces. The key is that when we want to identify \(\apd_{s_{\infty}}(p)\) and \(\apd_{s_{\infty}}(q)\) we get stuck, since the two have different types --- one is a path from \(p\tr s(x)\), the other from \(q\tr s(x)\) for some \(x\). One of those then needs to be adjusted by exactly the path that we are trying to manifest.

Coherence of the right homotopy follows from coherence of \(s_{\infty}\), as it involves the homotopies \(K_n\), \(\kapn_B^n\), and the computation rules of \(s_{\infty}\).

The necessary homotopy algebra is not interesting, so it is not reproduced here, but the full proof is available in the formalization.
\end{proof}

\begin{figure}
\begin{subfigure}{.34\textwidth}
\begin{tikzcd}[row sep=1ex]
  Q_n
  \arrow[rr, "\kapn_B^n \tr"]
  \arrow[dr, "\id"', near end]
  && Q_{n+1}
  \arrow[dl, "\id", near end] \\
  & Q \\
  B_n
  \arrow[uu, "s_n"]
  \arrow[rr, "b_n", near start]
  \arrow[dr, "\inn_B^n"']
  && B_{n+1}
  \arrow[uu, "s_{n+1}"']
  \arrow[dl, "\inn_B^{n + 1}"] \\
  & B_{\infty}
  \arrow[uu, "s_{\infty}"', near start, crossing over]
\end{tikzcd}
\vspace{.35ex}
\centering\caption{}
\label{subfig:prism-coh-dep}
\end{subfigure}
\begin{subfigure}{.35\textwidth}
\begin{tikzcd}[column sep=45pt]
  (Q \comp f_{\infty})_n \arrow[r, "C_n \tr"] & Q_n \\
  A_n
  \arrow[u, "s_{\infty} \comp f_{\infty} \comp \inn_A^n"]
  \arrow[r, "f_n"']
  & B_n \arrow[u, bend right, "s_n"']
  \arrow[u, bend left, "s_{\infty} \comp \inn_B^n"]
\end{tikzcd}
\centering\caption{}
\label{subfig:htpy-sf}
\end{subfigure}
\centering
\caption{\subref{subfig:prism-coh-dep} Coherence of the computation rules of \(s_{\infty}\), which are used for the coherence of the right homotopy in \subref{subfig:htpy-sf}.}
\end{figure}

\begin{cor}
\href{https://archive.vojtechstep.eu/zigzag-construction/synthetic-homotopy-theory.dependent-functoriality-sequential-colimits.html\#cubes-of-sections-of-pulled-back-descent-data-induce-homotopies-of-composed-maps-in-the-colimit}{\ExternalLink} Given the same data as in \autoref{lem:cube-compr}, we get a homotopy of dependent maps \(t_{\infty} \htpy s_{\infty} \comp f_{\infty}\).
\label{lem:improved-preserves-comp}
\end{cor}

\begin{proof}
By \autoref{lem:preserves-comp} \(s_{\infty} \comp f_{\infty}\) is homotopic to \((s_{\bullet} \comp f_{\bullet})_{\infty}\), so it suffices to construct a homotopy between the sections \(t\) and \((s_{\bullet} \compr f_{\bullet})\) to get a homotopy of the dependent maps they induce. This homotopy is provided exactly by \autoref{lem:cube-compr}.
\end{proof}

\begin{constr}
\href{https://archive.vojtechstep.eu/zigzag-construction/synthetic-homotopy-theory.dependent-functoriality-sequential-colimits.html\#composition-of-a-section-and-a-fiberwise-morphism}{\ExternalLink}
Given a morphism of sequential diagrams \(f_{\bullet} : A_{\bullet} \to B_{\bullet}\), a fiberwise map \(e\{a : A_{\infty}\}: P(a) \to Q(f_{\infty}(a))\), and a section \((t, K)\) of \(P_{\bullet}\), we construct the section \(e_{\bullet} \compl t_{\bullet}\) of \((Q \comp f_{\infty})_{\bullet}\)\label{compl}.
Once again, the symbol ``\(\,\compl\)'' is supposed to evoke the shape of the composition. The maps \((e_{\bullet} \compl t_{\bullet})_n\) are defined to be \((e \comp t_n) : (a : A_n) \to Q(f_{\infty}(\inn_A^{n}(a)))\). The coherences are obtained by vertically composing the diagrams in \autoref{fig:comp-coh}.
\label{defn:comp'}
\end{constr}

\begin{figure}
\begin{subfigure}{.5\textwidth}
\begin{tikzcd}
  P_n(a) \arrow[d, "e"'] \arrow[r, "\kapn_A^n(a) \tr"] & P_{n+1}(a_n(a)) \arrow[d, "e"] \\
  (Q \comp f_{\infty})_n(a) \arrow[r, "\kapn_A^n(a) \tr"'] & (Q \comp f_{\infty})_{n+1}(a_n(a))
\end{tikzcd}
\centering\caption{}
\label{subfig:coh-over}
\end{subfigure}
\begin{subfigure}{.4\textwidth}
\begin{tikzcd}
  P_n \arrow[r, "\kapn_A^n \tr"] & P_{n+1} \\
  A_n \arrow[u, "t_n"] \arrow[r, "a_n"'] & A_{n+1} \arrow[u, "t_{n+1}"']
\end{tikzcd}
\centering\caption{}
\label{subfig:coh-sect}
\end{subfigure}
\centering
\caption{\label{fig:comp-coh}Coherences of \(e_{\bullet} \compl t_{\bullet}\) are obtained by pasting \subref{subfig:coh-over} the square witnessing commutativity of transports and fiberwise maps, on top of \subref{subfig:coh-sect} the coherence square \(K_n\) of \(t_{\bullet}\).}
\end{figure}

\begin{lem}
\href{https://archive.vojtechstep.eu/zigzag-construction/synthetic-homotopy-theory.dependent-functoriality-sequential-colimits.html\#sequential-colimits-preserve-compositions-of-sections-and-fiberwise-morphisms}{\ExternalLink}
In the context of \autoref{defn:comp'}, whenever \(P\) and \(Q\) are in the same universe and \(e\) is a family of equivalences, the dependent map \((e_{\bullet} \compl t_{\bullet})_{\infty} : (a : A_{\infty}) \to Q(f_{\infty}(a))\) is homotopic to \(e \comp t_{\infty}\).
\label{lem:preserves-fib-comp}
\end{lem}

\begin{proof}
Since fiberwise equivalences of type families in the same universe characterize their identity types, we may assume \(P \judeq (Q \comp f_{\infty})\) and \(e\{a\} \judeq \id\). Then we are asked to show a homotopy \((\id_{\bullet} \compl t_{\bullet})_{\infty} \htpy t_{\infty}\), which is a homotopy between induced dependent functions, so it suffices to show a homotopy of their underlying sections of dependent diagrams. This homotopy is constructed by taking \(\reflhtpy : t_n \htpy t_n\) on the maps, and calculating that commutativity of transport and fiberwise identity is homotopic to \(\reflhtpy : (\kapn_A^n(a) \tr) \htpy (\kapn_A^n(a) \tr)\), so the coherence part requires a trivial coherence between \(K_n\) and \(K_n\).
\end{proof}

\begin{rem}
We limit the type families to be in the same universe and require \(e\) to be an equivalence in order to apply induction on fiberwise equivalences. This result is expected to hold without these restrictions, however they simplify the formalization of this proof --- both sides take on the form \(\blank_{\infty}\), so the problem reduces to a unit law of sections of dependent diagrams, whose map part computes judgmentally, and no higher path algebra is needed. In fact, the proof of the unit law is shorter than the definition of the composition \(e_{\bullet} \compl t_{\bullet}\)!
\label{remark:equiv-good}
\end{rem}

\begin{proof}[Proof of \autoref{lem:preserves-cubes}]
By \autoref{lem:preserves-fib-comp} we have a homotopy \(e \comp t_{\infty} \htpy (e_{\bullet} \compl t_{\bullet})_{\infty}\). To get the remaining homotopy \((e_{\bullet} \compl t_{\bullet})_{\infty} ~ s_{\infty} \comp f_{\infty}\) we use \autoref{lem:improved-preserves-comp} --- by definition of \(e_n\), the homotopies \(F_n\) have the correct type, taking \(e_{\bullet} \compl t_{\bullet}\) for the left section. To finish the proof, some mechanical path algebra is necessary to transform the cubes we are given to the ones expected by \autoref{lem:improved-preserves-comp}, since in the cubes we have the commuting squares involving \(e\) are part of the back square, not the top square.
\end{proof}
\section{Path Spaces of Pushouts}
\label{sec:idpo}
Wärn describes an explicit construction of identity types of pushouts \cite{War23}. He does so by fixing an element \(a_0 : A\), and then defining type families (with infix notation) \(a_0 \rightsquigarrow_{\infty} a\) and \(a_0 \rightsquigarrow_{\infty} b\), such that for any \(a : A\) and \(b : B\), there are equivalences \((\inl(a_0) = \inl(a)) \simeq (a_0 \rightsquigarrow_{\infty} a)\) and \((\inl(a_0) = \inr(b)) \simeq (a_0 \rightsquigarrow_{\infty} b)\). The type families are defined by gradual approximations of the identity types, \(a_0 \rightsquigarrow_{2n} a\) and \(a_0 \rightsquigarrow_{2n + 1} b\). If one thinks of the standard pushout \(\PO{S}{A}{B}\) as a coproduct \(A + B\) with added paths from \(f(s)\) to \(g(s)\), then \(a_0 \rightsquigarrow_{2n} a\) describes the type of identifications between \(\inl(a_0)\) and \(\inl(a)\), provided that we can pass from the \(A\) component to the \(B\) component up to \(n\) times, and similarly for \(a_0 \rightsquigarrow_{2n + 1} b\). The full identity types are then constructed by removing the upper bound on the number of steps, by taking the sequential colimit.

The two type families are related --- if one can get from \(\inl(a_0)\) to \(\inl(f s)\) in \(n\) crossings, then one can get from \(\inl(a_0)\) to \(\inr(g s)\) in \(n + 1\) crossings, and similarly in reverse. We can formally encode this relationship in a structure called a ``zigzag'' between sequential diagrams. We begin by defining general zigzags of sequential diagrams and their behavior in the colimit. Then we define the type families of approximations of identity types, and a zigzag between them. We finish by showing that the induced type families and equivalence form an identity system of descent data, which gives us the desired equivalences by applying \autoref{thm:unique-uniqueness-id-system}.

\begin{defi}
Given sequential diagrams \((A, a)\) and \((B, b)\), a \textbf{zigzag}~\href{https://archive.vojtechstep.eu/zigzag-construction/synthetic-homotopy-theory.zigzags-sequential-diagrams.html\#a-zigzag-between-sequential-diagrams}{\ExternalLink} between them is a quadruple \((f, g, U, L)\), where \(f_n : A_n \to B_n\) and \(g_n : B_n \to A_{n+1}\) are families of maps, and \(U_n : a_n \htpy (g_n \comp f_n)\) and \(L_n : b_n \htpy (f_{n+1} \comp g_n)\) are families of commuting triangles.
\end{defi}

A zigzag \(z \defeq (f, g, U, L)\) can be seen as the following sequence of adjacent triangles:
\begin{center}
\begin{tikzcd}[column sep=small]
  A_0 \arrow[rr, "a_0"]
  \arrow[dr, "f_0"', ""{name=F0, anchor=center}]
  & \arrow[d, phantom, yshift=2pt, "U_0"]
  & A_1 \arrow[rr, "a_1"]
  \arrow[dr, "f_1"]
  &&[-10pt] \cdots \\
  & B_0 \arrow[rr, "b_0"']
  \arrow[ur, "g_0"']
  & \arrow[u, phantom, yshift=-2pt, xshift=2pt, "L_0"]
  & B_1 \arrow[rr, "b_1"']
  && \cdots.
\end{tikzcd}
\end{center}

By forgetting the first triangle and turning the figure upside down, we get a new zigzag \((g, (n \mapsto f_{n+1}), L, (n \mapsto U_{n+1}))\), this time between \(B_{\bullet}\) and \(A_{\bullet+1}\). This new zigzag is called the \textbf{half-shift}~\href{https://archive.vojtechstep.eu/zigzag-construction/synthetic-homotopy-theory.zigzags-sequential-diagrams.html\#half-shifts-of-zigzags-of-sequential-diagrams}{\ExternalLink} of \(z\).

\begin{constr}
\href{https://archive.vojtechstep.eu/zigzag-construction/synthetic-homotopy-theory.zigzags-sequential-diagrams.html\#morphisms-of-sequential-diagrams-induced-by-zigzags-of-sequential-diagrams}{\ExternalLink}
A zigzag induces a morphism of diagrams \((f, H) : A_{\bullet} \to B_{\bullet}\), where the squares are constructed by pasting triangles,
\(H_n \defeq (L_n \rwhisk f_n) \hconcat (f_{n+1} \lwhisk U_n^{-1})\). Then the induced function between colimits is \(f_{\infty} : A_{\infty} \to B_{\infty}\).

The half-shift induces the inverse morphism of diagrams \(g_{\bullet} : B_{\bullet} \to A_{\bullet+1}\). Note that the cocone \(A_{\infty}\) is a cocone under \(A_{\bullet}\), and by forgetting the first triangle we get a cocone under \(A_{\bullet+1}\). By \cite[Lemma 3.6]{SvDR20} this cocone equips the type \(A_{\infty}\) with the structure of a colimit of \(A_{\bullet+1}\), so the inverse map induced by \(g_{\bullet}\) is \(g_{\infty} : B_{\infty} \to A_{\infty}\).
\label{constr:zigzag-hom}
\end{constr}

It deserves the name ``inverse'', because we now show that \(g_{\infty}\) is an inverse of \(f_{\infty}\).

\begin{thm}
\href{https://archive.vojtechstep.eu/zigzag-construction/synthetic-homotopy-theory.zigzags-sequential-diagrams.html\#zigzags-of-sequential-diagrams-induce-equivalences-of-sequential-colimits}{\ExternalLink}
Consider a zigzag \((f, g, U, L)\) between \((A, a)\) and \((B, b)\). Then \(g_{\infty}\) and \(f_{\infty}\) are mutually inverse equivalences.
\label{thm:equiv-zigzag}
\end{thm}

\begin{figure}
\begin{tikzcd}
  A_n \arrow[r, "a_n"]
  \arrow[d, "f_n"', ""{name=FN}]
  \arrow[dd, rounded corners, to path={-- ([xshift=-20pt]\tikztostart.west) |- (\tikztotarget)\tikztonodes}, pos=0.25, "a_n"', ""{name=AN}]
  & A_{n + 1}
  \arrow[d, "f_{n + 1}", ""{name=FN1}]
  \arrow[from=FN, phantom, "U_n^{-1}"]
  \arrow[dd, rounded corners, to path={-- ([xshift=25pt]\tikztostart.east) |- (\tikztotarget)\tikztonodes}, pos=0.25, "a_{n+1}", ""{name=AN1}]
 \\
  B_n \arrow[r]
  \arrow[ur]
  \arrow[d, "g_n"', ""{name=GN}]
  \arrow[from=AN, phantom, "U_n"]
  \arrow[to=FN1, phantom, "L_n"]
  & B_{n + 1}
  \arrow[d, "g_{n + 1}", ""{name=GN1}]
  \arrow[from=GN, phantom, "L_n^{-1}"]
  \arrow[to=AN1, phantom, pos=0.4, "U_{n+1}^{-1}"]
\\
  A_{n + 1} \arrow[r, "a_{n + 1}"']
  \arrow[ur]
  \arrow[to=GN1, shift right=2pt, phantom, "U_{n + 1}"]
  & A_{n + 2}
\end{tikzcd}
\caption{\label{fig:zigzag-inv-comp}Coherences of the composition of a zigzag and its half-shift.}
\end{figure}

\begin{proof}
The proof proceeds mostly by path algebra, but it also takes advantage of half-shifts: we show that \(g_{\infty}\) is a retraction of \(f_{\infty}\), and then apply the same proof to the half-shift, which shows that a different map is a retraction of \(g_{\infty}\), so \(g_{\infty}\) has both a section and a retraction.

By functoriality \cite[Lemma 3.5]{SvDR20}, we have a homotopy \(g_{\infty} \comp f_{\infty} \htpy (g_{\bullet} \comp f_{\bullet})_{\infty}\). The way \(A_{\infty}\) is constructed as a colimit of \(A_{\bullet+1}\) means the map of colimits induced by the shifting morphism \((a, \reflhtpy) : A_{\bullet} \to A_{\bullet+1}\) is the identity on \(A_{\infty}\): composing it with the cocone under \(A_{\bullet+1}\) recovers the original cocone \(A_{\infty}\). We continue by showing that the morphisms \((a, \refl)\) and \(g_{\bullet} \comp f_{\bullet}\) are homotopic. The homotopy of maps is given by \(U_n\), and the coherence amounts to showing that the diagram in \autoref{fig:zigzag-inv-comp} is homotopic to the reflexive homotopy. This follows by path algebra by canceling out the pairs \(U_n, U_n^{-1}\), \(L_n, L_n^{-1}\), and \(U_{n+1}, U_{n+1}^{-1}\). This concludes the homotopy \(g_{\infty} \comp f_{\infty} \htpy \id\). To construct the retraction of \(g_{\infty}\), apply the above argument to the half shift of the zigzag --- the map \(g_{\infty}\) then appears in the other position as \(f'_{\infty} \comp g_{\infty} \htpy \id\), where \(f'_{\infty}\) is the map induced by the full shift (double half shift).
\end{proof}
\subsection{Zigzag construction of path spaces of pushouts}
\label{sec:org458a43b}

The construction of identity types below replicates the zigzag construction of Wärn \cite{War23}. The reader is invited to refer to Wärn's paper for an exposition of the definition, while we focus on encoding the construction and verifying its correctness in Agda.

For the remainder of the paper, consider a span diagram
\begin{tikzcd}[cramped]
A & S \arrow[l, "f"'] \arrow[r, "g"] & B
\end{tikzcd}
whose pushout's path spaces we want to characterize, and a basepoint \(a_0 : A\). We represent the zigzag construction as descent data. To construct it, we need two type families \(P_A : A \to \UU\) and \(P_B : B \to \UU\), which we define as sequential colimits of certain diagrams \(P_A^{\bullet}\) and \(P_B^{\bullet}\), and a family of equivalences \(P_S\{s\} : P_A(f s) \simeq P_B(g s)\), obtained by constructing a zigzag.

The data we need to construct is thus a sequence of type families \(P_A^n\) over \(A\) with a family of connecting maps \(\incl_A^n\{a : A\} : P_A^n(a) \to P_A^{n+1}(a)\), similarly for \(P_B^n\) and \(\incl_B^n\), and for every \(s : S\) a pair of families of maps \(\blank \concatpre_n s : P_A^n(f s) \to P_B^{n + 1}(g s)\) and \(\blank \concatpre_n \overline{s} : P_B^n(g s) \to P_A^n(f s)\), with some homotopies between them, all of which are indexed by \(n : \N\). Note in particular that the bar over \(s\) is part of the notation for the last map, not an operation on \(S\) by itself.

The intended meaning of \(P_A^n(a)\) and \(P_B^n(b)\) is still the types of paths from \(\inl(a_0)\) to \(\inl(a)\) or \(\inr(b)\), respectively, with at most \(n\) jumps to the \(B\) component. The maps \(\incl_A^n\) and \(\incl_B^n\) then represent the inclusion maps of ``a path with at most \(n\) jumps'' to ``a path with at most \(n+1\) jumps'', and the maps \(\blank \concatpre_n s\) and \(\blank \concatpre_n \overline{s}\) stand for adjoining a jump represented by the path \(\glue(s)\) between \(\inl(fs)\) and \(\inr(gs)\), respectively to and from the \(B\) component.

The construction proceeds by induction on \(n\), with various interdependencies between the definitions. We first present a naïve description of the construction, but then argue for a different induction motive with better computational properties.

Guided by the intuition, we want to take \(P_A^0(a)\) to be the identity type \((a_0 = a)\), \(P_B^0(b)\) to be the empty type \(\0\), and \(\blank \concatpre_0 \overline{s}\) to be the unique map out of the empty type. But then to define \(\blank \concatpre_0 s\), we need to know what \(P_B^1(gs)\) is. The intention is to define \(P_A^{n+1}\) and \(P_B^{n+1}\) as pushouts, specified by span diagrams which use \(\blank \concatpre_n s\) and \(\blank \concatpre_n \overline{s}\), respectively, and recursively define \(\blank \concatpre_{n+1} \overline{s}\) and \(\blank \concatpre_n s\) as the right point constructors of those pushouts. This nontrivial dependence schema is depicted in \autoref{fig:zigzag-deps}. We also want to consider computational behavior: if the proof assistant allowed us to transcribe this description and define everything together by simple induction on \(\N\), we would end up with \(\blank \concatpre_0 s\) and \(\blank \concatpre_{n+1} s\) with the same body, but in different cases of the induction, so it would not hold that \(\blank \concatpre_n s\) is judgmentally the right point constructor of \(P_B^{n+1}(g s)\) for all \(n\).

\begin{figure}
\begin{align*}
  \cdot &\vdash P_B^0 & P_B^n, P_A^n, \blank \concatpre_n \overline{s} &\vdash P_B^{n + 1} \\
  \cdot &\vdash P_A^0 & P_A^n, P_B^{n + 1}, \blank \concatpre_n s &\vdash P_A^{n + 1} \\
  \cdot &\vdash \blank \concatpre_0 \overline{s} & P_A^{n + 1} &\vdash \blank \concatpre_{n + 1} \overline{s} \\
  P_B^1 &\vdash \blank \concatpre_0 s & P_B^{n + 1} &\vdash \blank \concatpre_n s
\end{align*}
\caption{\label{fig:zigzag-deps}Dependencies between definitions in the zigzag construction. Note that with this naïve recursive motive we artificially split the definition of \(\blank \concatpre_n s\) into two cases. In the formalization we remove \(\blank \concatpre_n s\) from the motive, inlining its definition in the construction of \(P_A^{n+1}\), and then define \(\blank \concats{n} s\) after induction is done. This approach gives \(\blank \concats{n} s\) a uniform definition for all \(n\).}
\end{figure}

The motive which we proceed to formalize removes the \(\blank \concatpre_n s\) component altogether. In the construction itself it is only used to define \(P_A^{n + 1}\), where it can be replaced by a direct reference to the right point constructor of the pushout \(P_B^{n + 1}(g s)\), which is already defined by the time we need to define \(P_A^{n + 1}\). Then \(\blank \concats{n} s\) can be defined after the construction as the right point constructor at every stage, without induction, removing code duplication and giving it the right computational behavior. We also want to refer to the span diagrams defining \(P_B^{n + 1}\) and \(P_A^{n + 1}\) later in the code, hence we also remember those in the construction.

\begin{defi}
Given a natural number \(n\), define the type of \textbf{zigzag construction data}~\href{https://archive.vojtechstep.eu/zigzag-construction/synthetic-homotopy-theory.zigzag-construction-identity-type-pushouts.html\#the-motive-of-zigzag-constructions-data}{\ExternalLink} at stage \(n\) to be the type of quadruples \((P_B^n, P_A^n, \blank \concatpre_n \overline{s}, D)\), where \(P_B^n\) is a type family over \(B\), \(P_A^n\) is a type family over \(A\), \(\blank \concatpre_n \overline{s} : P_B^n(g s) \to P_A^n(f s)\) is a family of maps indexed by \(s : S\), and \(D\) is an element of the unit type if \(n = 0\), or of the type of pairs \((\T_B^n, \T_A^n)\), where \(\T_B^n\) is a family of span diagrams indexed by \(B\) and \(\T_A^n\) is a family of span diagrams indexed by \(A\), if \(n\) is nonzero.
\end{defi}

This type can be inhabited for all \(n : \N\).

\begin{constr}
\href{https://archive.vojtechstep.eu/zigzag-construction/synthetic-homotopy-theory.zigzag-construction-identity-type-pushouts.html\#zigzag-construction-data}{\ExternalLink}
Construct an inhabitant of the type of zigzag construction data for every stage \(n\) by induction.

For the zero case, use \(P_B^0(b) \defeq \0\), \(P_A^0(a) \defeq (a_0 = a)\), define \(\blank \concatpre_0 \overline{s}\) by \(\exf\), and inhabit \(D^0\) by the unique element of the unit type.

For the successor case \(n + 1\), first construct the families of span diagrams \(\T_B^{n + 1}\). For an element \(b : B\), define \(\T_B^{n + 1}(b)\) to be the span diagram
\begin{center}
\begin{tikzcd}
  P_B^n(b)
  & \Sigma (s : S).\, (r : b = g s) \times P_B^n(b)
  \arrow[l, "\pr_3"']
  \arrow[r, "\chi"]
  & \Sigma (s : S).\, (r : b = g s) \times P_A^n(f s),
\end{tikzcd}
\end{center}
where \(\chi\) sends \((s, r, p)\) to \((s, r, (r\tr p) \concatpre_n \overline{s})\). Take \(P_B^{n + 1}(b)\) to be the standard pushout of this diagram. Analogously, for an element \(a : A\), define \(\T_A^{n + 1}(a)\) to be the span diagram
\begin{center}
\begin{tikzcd}
  P_A^n(a)
  & \Sigma (s : S).\, (r : a = f s) \times P_A^n(a)
  \arrow[l, "\pr_3"']
  \arrow[r, "\theta"]
  & \Sigma (s : S).\, (r : a = f s) \times P_B^{n + 1}(g s)
\end{tikzcd}
\end{center}
where the map \(\theta\) takes \((s, r, p)\) to \((s, r, \inr(s, \refl, r \tr p))\), using the right point constructor \(\inr\) into the pushout \(P_B^{n + 1}(g s)\). Then define \(P_A^{n + 1}(a)\) to be the standard pushout of \(\T_A^{n + 1}(a)\). Finally, define \(p \concatpre_{n + 1} \overline{s}\) to be \(\inr(s, \refl, p)\) using the right point constructor into \(P_A^{n + 1}(f s)\).
\label{constr:zigzag-spans}
\end{constr}

We keep using the names \(\PB{n}\)\mlabel{PBn}, \(\PA{n}\)\mlabel{PAn}, and \(\blank \concatinvs{n} \overline{s}\)\mlabel{concatinvs} for the corresponding elements of this canonical construction.

\begin{constr}
\href{https://archive.vojtechstep.eu/zigzag-construction/synthetic-homotopy-theory.zigzag-construction-identity-type-pushouts.html\#inverse-bridge-concatenation}{\ExternalLink}
For every stage \(n : \N\) and element \(s : S\), define the map \(\blank \concats{n} s\)\mlabel{concats} from \(\PA{n}(f s)\) to \(\PB{n + 1}(g s)\) to send \(p\) to \(\inr(s, \refl, p)\), where \(\inr\) is the right point constructor of \(\PB{n + 1}(g s)\).
\end{constr}

We may now construct the sequential diagrams of approximations of the type families \((\inl(a_0) = \inr(b))\) and \((\inl(a_0) = \inl(a))\).

\begin{constr}
\href{https://archive.vojtechstep.eu/zigzag-construction/synthetic-homotopy-theory.zigzag-construction-identity-type-pushouts.html\#right-family}{\ExternalLink}
Given an element \(b : B\), define the sequential diagram \(\PBdiag{\bullet}(b)\)\label{PBdiag} to be the diagram
\begin{tikzcd}[cramped]
  \PB{0}(b)
  \arrow[r, "\inclB{0}"]
  & \PB{1}(b)
  \arrow[r, "\inclB{1}"]
  & \PB{2}(b)
  \arrow[r, "\inclB{2}"]
  & \cdots,
\end{tikzcd}
where the maps \(\inclB{n}\)\label{inclB} are the left point constructors \(\inl\) of \(\PB{n + 1}(b)\). Denote its sequential colimit \(\PBinf(b)\)\label{PBinf}, with point constructors \(\inB{n}\)\label{inB} and path constructors \(\kapB{n}\)\label{kapB}.
\label{defn:path-to-b}
\end{constr}

\begin{constr}
\href{https://archive.vojtechstep.eu/zigzag-construction/synthetic-homotopy-theory.zigzag-construction-identity-type-pushouts.html\#left-family}{\ExternalLink}
Given an element \(a : A\), define the sequential diagram \(\PAdiag{\bullet}(a)\)\label{PAdiag} to be the diagram
\begin{tikzcd}[cramped]
  \PA{0}(a)
  \arrow[r, "\inclA{0}"]
  & \PA{1}(a)
  \arrow[r, "\inclA{1}"]
  & \PA{2}(a)
  \arrow[r, "\inclA{2}"]
  & \cdots,
\end{tikzcd}
where the maps \(\inclA{n}\)\label{inclA} are the left point constructors \(\inl\) of the pushouts defining \(\PA{n + 1}(a)\). Denote its sequential colimit \(\PAinf(a)\)\label{PAinf}, with point constructors \(\inA{n}\)\label{inA} and path constructors \(\kapA{n}\)\label{kapA}.
\label{defn:path-to-a}
\end{constr}

\begin{rem}
We want to be very careful about the numeric indices of \(P_B\). The type \(\PB{0}(b)\) is essentially irrelevant, since data over it can always be defined by \(\exf\). We could have just as well shifted the sequence and had \(\PB{0}(b)\) be the pushout we currently call \(\PB{1}(b)\). However, this numbering would not align with the intended intuition. Additionally, since \(\PA{n}\) and \(\PB{n}\) are defined together, it is good for uniformity to have constructions at the zeroth index be atomic, and constructions at successor indices be pushouts.

In general there is no map from \(\PA{0}\) to \(\PB{0}\), so when constructing the zigzag below it needs to be between the sequences \(\PAdiag{\bullet}(f s)\) and \(\PBdiag{\bullet + 1}(g s)\). The convention we will follow is using \(\PBdiag{\bullet}\) when data needs to be defined mutually over \(\PA{n}\) and \(\PB{n}\), for uniformity, but then mostly regarding \(\PBinf(b)\) as the colimit of \(\PBdiag{\bullet + 1}\) via \cite[Lemma 3.6]{SvDR20}, since that is how we are able to construct a zigzag.

One option to sidestep this non-uniform approach is to take the sequential diagram defining \(\PBinf(b)\) to start at \(\PB{1}(b)\). The author judged this presentation to be more confusing, since then the domain of \eg \(\inB{n}\) would be \(\PB{n+1}\), and similarly a type family \(Q\) over \(\PBinf(b)\) would induce a dependent sequential diagram in such a way that \(Q_n\) would lie over \(\PB{n+1}(b)\).
\end{rem}

\begin{constr}
\href{https://archive.vojtechstep.eu/zigzag-construction/synthetic-homotopy-theory.zigzag-construction-identity-type-pushouts.html\#zigzag-between-left-and-right-families}{\ExternalLink}
Given an element \(s : S\), construct the zigzag between \(\PAdiag{\bullet}(f s)\) and \(\PBdiag{\bullet+1}(g s)\) as
\begin{center}
\begin{tikzcd}[column sep=tiny]
  (a_0 = f s) \arrow[rr, "\inclA{0}"]
  \arrow[dr, "\blank \concats{0} s"']
  && \PA{1}(f s) \arrow[rr, "\inclA{1}"]
  \arrow[dr, "\blank \concats{1} s"]
  && \cdots \\
  & \PB{1}(g s) \arrow[rr, "\inclB{1}"']
  \arrow[ur, "\blank \concatinvs{1} \overline{s}"']
  && \PB{2}(g s) \arrow[rr, "\inclB{2}"']
  && \cdots,
\end{tikzcd}
\end{center}
where the triangles are the partially applied path constructors
\begin{alignat*}{2}
  \hspace{-1em}\glue(s, \refl, \blank) &: \inclA{n} \htpy (\blank \concats{n} s) \concatinvs{n + 1} \overline{s} &\hspace{2em}
  \glue(s, \refl, \blank) &: \inclB{n+1} \htpy (\blank \concatinvs{n+1} \overline{s}) \concats{n+1} s
\end{alignat*}
of \(\PA{n + 1}(f s)\) and \(\PB{n + 2}(g s)\), respectively.

In the context of this zigzag, we refer to the triangles as only \(\glueA{n}\)\label{glueA} and \(\glueB{n+1}\)\label{glueB}, not mentioning the \(s\) and \(\refl\) arguments.
\label{defn:zigzag-zigzag}
\end{constr}

\begin{constr}
Define the \textbf{zigzag construction}~\href{https://archive.vojtechstep.eu/zigzag-construction/synthetic-homotopy-theory.zigzag-construction-identity-type-pushouts.html\#zigzag-construction-descent-data}{\ExternalLink} descent data \((\PAinf, \PBinf, \blank \concatinf s)\), where the type families are \autoref{defn:path-to-a} and \autoref{defn:path-to-b}, respectively, and the family of equivalences \(\blank \concatinf s : \PAinf(f s) \simeq \PBinf(g s)\)\mlabel{concatinf}, parameterized by \(s : S\), is induced by \autoref{defn:zigzag-zigzag} using \autoref{thm:equiv-zigzag}.

Additionally, this descent data is pointed with the element \(\inA{0}(\refl_{a_0}) : \PAinf(a_0)\), which we call \(\reflinf\)\mlabel{refl-infty}.
\end{constr}
\subsection{Correctness of the zigzag construction}
\label{sec:org7cfd7c9}

The rest of the paper proves that the zigzag descent data is an identity system pointed at \(\reflinf\), verifying that the construction characterizes the path spaces of pushouts. To that end, we consider arbitrary pointed descent data over its total span, and construct a section. For readability, we curry all the components and make some arguments implicit.

Namely, in the remainder of this section consider type families \(Q_{\Sigma A}\{a : A\} : \PAinf(a) \to \UU\), \(Q_{\Sigma B}\{b : B\} : \PBinf(b) \to \UU\), equivalences \(Q_{\Sigma S} \{s : S\} \{p : \PAinf(fs)\}: Q_{\Sigma A}(p) \simeq Q_{\Sigma B}(p \concatinf s)\),
and a point \(q_0 : Q_{\Sigma A}(\reflinf)\). The goal is to conjure a section, \ie define a pair of dependent functions \(t_A^{\infty}\{a : A\} : (p : \PAinf(a)) \to Q_{\Sigma A}(p)\), \(t_B^{\infty}\{b : B\} : (p : \PBinf(b)) \to Q_{\Sigma B}(p)\),
and a family of identifications \(t_S^{\infty}\{s : S\} : (p : \PAinf(f s)) \to Q_{\Sigma S}(t_A^{\infty}(p)) = t_B^{\infty}(p \concatinf s)\)

Recall that \autoref{lem:functoriality} gives us the family of coherent homotopies \(C_n\{s : S\}(p : \PA{n}(f s)) : \inA{n}(p) \concatinf s = \inB{n+1}(p \concats{n} s)\), and the type families \(Q_{\Sigma A}\{a\}\) and \(Q_{\Sigma B}\{b\}\) induce dependent sequential diagrams \(Q_{\Sigma A}^{\bullet}\{a\}\) over \(\PAdiag{\bullet}(a)\) and \(Q_{\Sigma B}^{\bullet}\{b\}\) over \(\PBdiag{\bullet}(b)\), respectively.

In order to define the dependent maps \(t_A^{\infty}\) and \(t_B^{\infty}\), we will proceed by induction on their respective arguments \(p\), which are elements of sequential colimits. Using the dependent universal property, this amounts to providing dependent maps \(t_A^n \{a : A\} (p : \PA{n}(a)) : Q_{\Sigma A}^n(p)\) and coherences \(K_A^n \{a : A\} (p : \PA{n}(a)) : \kapA{n}(p) \tr t_A^n(p) = t_A^{n + 1}(\inclA{n}(p))\), and analogously for \(t_B^n\) and \(K_B^n\). The coherence \(t_S^{\infty}\) will then be constructed by building coherence cubes relating those sections, and applying \autoref{lem:preserves-cubes}.

Let us begin by defining the maps. We will once again start from a semi-formal naïve construction, point out the computational difficulties, and then give the formal version.

The maps will be defined together by induction on \(n\). In the zero case, \(t_A^0\{a\}\) eliminates \(p : (a_0 = a)\) by path induction and returns the provided basepoint \(q_0 : Q_{\Sigma A}^0(\refl)\), and \(t_B^0(b)\) eliminates \(p : \0\) by \(\exf\). In the successor case, we are eliminating out of the pushouts \(\PA{n+1}(a)\) and \(\PB{n+1}(b)\) using the dependent universal property. Recall that this means providing dependent cocones under the spans defined in \autoref{constr:zigzag-spans}. Since functions out of pushouts do not compute even on point constructors in our setting, we differentiate the induced map, \eg \(t_B^{n+1}\), from the two maps in its defining dependent cocone, \eg \(\overline{t_B^{n+1}}(\inl(\blank))\) and \(\overline{t_B^{n+1}}(\inr(\blank))\). We may visualize the problem in three dimensions using dependent diagrams, as indicated in \autoref{subfig:map-prisms} --- the actions on point constructors are defined using maps from previous stages, the fiberwise maps \(Q_{\Sigma S}^n\{s, p\} : Q_{\Sigma A}^n(p) \to Q_{\Sigma B}^{n+1}(p \concats{n} s)\) induced by \(Q_{\Sigma S}\) as in \autoref{defn:squares-over}, and the fiberwise equivalences \(\PhiDiag{n}\{s, p\} : Q_{\Sigma B}^{n+1}(p) \to Q_{\Sigma A}^{n+1}(p \concatinvs{n+1} \overline{s})\)\mlabel{PhiDiag} defined as
\begin{tikzcd}
  Q_{\Sigma B}^{n+1}(p)
  \arrow[r, "\kapB{n+1}(p) \tr" outer sep=2pt] &
  Q_{\Sigma B}^{n+2}(\inclB{n+1}(p))
  \arrow[r, "\glueB{n+1}(p) \tr" outer sep=3pt] &[10pt]
  Q_{\Sigma B}^{n+2}((p \concatinvs{n+1} \overline{s}) \concats{n+1} s)
  \arrow[r, "(Q_{\Sigma S}^{n+1})^{-1}"] &[5pt]
  Q_{\Sigma A}^{n+1}(p \concatinvs{n+1} \overline{s}).
\end{tikzcd}

\begin{figure}
\begin{tikzcd}[column sep=0, row sep=small]
  & Q_{\Sigma A}^n \arrow[dr, "Q_{\Sigma S}^n"] \arrow[from=dd, "t_A^n", very near start] \\
  Q_{\Sigma B}^n \arrow[rr, "\kapB{n} \tr" pos=0.2, crossing over]
  & & Q_{\Sigma B}^{n+1} \\
  & \PA{n}(f s) \arrow[dr, near start, "\blank \concats{n} s"] \\
  \PB{n}(b) \arrow[rr, "\inclB{n}"'] \arrow[uu, "t_B^n"]
  & & \PB{n+1} \arrow[uu, dotted, "t_B^{n+1}"']
\end{tikzcd}\hspace{5em}
\begin{tikzcd}[column sep=0, row sep=small]
  Q_{\Sigma A}^n \arrow[rr, "\kapA{n} \tr"]
  & & Q_{\Sigma A}^{n+1} \\
  & Q_{\Sigma B}^{n+1} \arrow[ur, "\PhiDiag{n}"'] \\
  \PA{n}(a) \arrow[rr, "\inclA{n}" pos=0.2] \arrow[uu, "t_A^n"]
  & & \PA{n+1} \arrow[uu, dotted, "t_A^{n+1}"'] \\
  & \PB{n+1}(g s) \arrow[ur, "\blank \concatinvs{n+1} \overline{s}"'] \arrow[uu, crossing over, "t_B^{n+1}" pos=0.1]
\end{tikzcd}
\centering
\caption{\label{subfig:map-prisms}Construction of the dependent maps \(t_B^{n+1}\) and \(t_A^{n+1}\). The type families at bottom right do not have an index, since each vertical square requires a different one.}
\end{figure}

To define \(t_B^{n + 1}\), we need to provide the dependent maps \(\overline{t_B^{n + 1}}\{b\}(\inl(\blank)) : (p : \PB{n}(b)) \to Q_{\Sigma B}^{n+1}(\inclB{n}(p))\) and \(\overline{t_B^{n + 1}}(\inr(s, \refl, \blank)) : (p : \PA{n}(f s)) \to Q_{\Sigma B}^{n+1}(p \concats{n} s)\).
Following \autoref{subfig:map-prisms}, the maps are defined by \(\overline{t_B^{n + 1}}(\inl(p)) \defeq \kapB{n}(p) \tr t_B^n(p)\) and \(\overline{t_B^{n + 1}}(\inr(s, \refl, p)) \defeq Q_{\Sigma S}^n(t_A^n(p))\). Similarly for \(t_A^{n + 1}\), we set \(\overline{t_A^{n + 1}}(\inl(p)) \defeq \kapA{n}(p) \tr t_A^n(p)\) and \(\overline{t_A^{n + 1}}(\inr(s, \refl, p)) \defeq \PhiDiag{n}(t_B^{n + 1}(p))\).

Next, we need to prove that those maps are coherent, \ie define families of identifications
\begin{align*}
  T_B^{n + 1}(s, \refl, \blank) : (p : \PB{n}(g s)) \to{}
  &\glueB{n}(p) \tr \overline{t_B^{n+1}}(\inclB{n}(p)) = \overline{t_B^{n+1}}((p \concatinvs{n} \overline{s}) \concats{n} s) \\
  T_A^{n+1}(s, \refl, \blank) : (p : \PA{n}(f s)) \to{}
  &\glueA{n}(p) \tr \overline{t_A^{n+1}}(\inclA{n}(p)) = \overline{t_A^{n+1}}((p \concats{n} s) \concatinvs{n+1} \overline{s}),
\end{align*}
which we write as \(T_B^{n+1}(p)\) and \(T_A^{n+1}(p)\), skipping the \(s\) and \(\refl\) arguments.

\begin{figure}
\begin{tikzcd}[column sep=-5pt, row sep=small]
  & Q_{\Sigma A}^{n+1} \arrow[dr, color=acc-blue, "Q_{\Sigma S}^{n+1}"] \arrow[from=dd, color=acc-blue, "t_A^{n+1}", very near start] \\
  Q_{\Sigma B}^{n+1} \arrow[rr, crossing over, color=acc-red, "/"{marking, near end}, "\kapB{n+1}\tr" near start]
  \arrow[ur, "\PhiDiag{n}"]
  & & Q_{\Sigma B}^{n+2} \\
  & \PA{n+1}(f s) \arrow[dr, near start, "\blank \concats{n+1} s"] \\
  \PB{n+1}(g s) \arrow[rr, "\inclB{n+1}"'] \arrow[uu, color=acc-red, "t_B^{n+1}"]
  \arrow[ur, color=acc-blue, near end, "\blank \concatinvs{n+1} \overline{s}"']
  & & \PB{n+2}(g s)
\end{tikzcd}\hspace{3em}
\begin{tikzcd}[column sep=-5pt, row sep=small]
  Q_{\Sigma A}^n \arrow[rr, color=acc-red, "/"{marking, near end}, "\kapA{n}\tr"]
  \arrow[dr, "Q_{\Sigma S}^n"']
  & & Q_{\Sigma A}^{n+1} \\
  & Q_{\Sigma B}^{n+1} \arrow[ur, color=acc-blue, "\PhiDiag{n}"'] \\
  \PA{n}(f s) \arrow[rr, "\inclA{n}" very near start] \arrow[uu, color=acc-red, "t_A^n"]
  \arrow[dr, color=acc-blue, "\blank \concats{n} s"']
  & & \PA{n+1}(f s) \\
  & \PB{n+1}(g s) \arrow[ur, "\blank \concatinvs{n+1} \overline{s}"'] \arrow[uu, crossing over, color=acc-blue, "t_B^{n+1}"', near end]
\end{tikzcd}
\caption{\label{fig:coh-prisms}Construction of the coherences \(T_B^{n+2}\) and \(T_A^{n+1}\). Slashed arrows indicate where an additional implicit transport is applied.}
\end{figure}

\begin{figure}
\begin{tikzcd}
  Q_{\Sigma A}^n(p)
  \arrow[rr, "\kapA{n}\tr"]
  \arrow[d, "Q_{\Sigma S}^n"']
  & & Q_{\Sigma A}^{n+1}(\inclA{n}(p))
  \arrow[d, "Q_{\Sigma S}^{n+1}"]
  \arrow[rr, "\glueA{n+1} \tr"]
  &[-17pt] & Q_{\Sigma A}^{n+1}(*)
  \\
  Q_{\Sigma B}^{n+1}(p \concats{n} s)
  \arrow[r, color=acc-blue, "\kapB{n+1} \tr"' outer sep=3pt]
  & Q_{\Sigma B}^{n+2}(\inclB{n+1}(p \concats{n} s))
  \arrow[r, "H_n \tr" outer sep=3pt]
  \arrow[dr, color=acc-blue, "\id"']
  & Q_{\Sigma B}^{n+2}(\inclA{n}(p) \concats{n+1} s)
  \arrow[d, "H_n^{-1} \tr"] \\
  & & Q_{\Sigma B}^{n+2}(\inclB{n+1}(p \concats{n} s))
 \arrow[rr, color=acc-blue, "\glueB{n+1} \tr"' outer sep=2pt]
  & & Q_{\Sigma B}^{n+2}(\dagger)
 \arrow[uu, color=acc-blue, "(Q_{\Sigma S}^{n+1})^{-1}"]
\end{tikzcd}
\caption{\label{fig:top-factor}Factorization of the top triangle of \(T_A^{n+1}\). The indices on the right are \(* \defeq (\inclA{n}(p) \concats{n} s) \concatinvs{n+1} \overline{s}\), and \(\dagger \defeq (p \concats{n} s) \concatinvs{n+1} \overline{s} \concats{n+1} s\).}
\end{figure}

The coherence \(T_B^{n+1}\) does further case analysis on \(n\). In the zero case, we are eliminating from the empty type \(\PB{0}(gs)\), so we pose \(T_B^1 \defeq \exf\). For the successor case, follow the left diagram in \autoref{fig:coh-prisms}. The type of \(T_B^{n+2}\) is the type of homotopies from the red composition to the blue composition \emph{over} the bottom triangle \(\glueB{n}\) --- the arrow is slashed to indicate that it is followed by an implicit transport along the homotopy drawn at the bottom of the diagram. We first construct the top triangle by observing that by the definition of \(\PhiDiag{n}\), it suffices to invoke the fact that \((Q_{\Sigma S}^{n+1})^{-1}\) is a section of \(Q_{\Sigma S}^{n+1}\), \ie there is an identification \(q = Q_{\Sigma S}^{n+1}((Q_{\Sigma S}^{n+1})^{-1}(q))\) for all \(q : Q_{\Sigma B}^{n+1}(p)\), which we instantiate with \(\glueB{n+1}(p) \tr (\kapB{n+1}(p) \tr t_B^{n+1}(p))\). The coherence is completed by the inverse of the right computation rule of \(t_A^{n+1}\), which fills the vertical square.

To construct \(T_A^{n+1}\), recall that we intend to finish the overall proof by \autoref{lem:preserves-cubes}. The cubes of sections will be obtained by pasting prisms whose outlines use \(T_B^{n+1}\) and \(T_A^{n+1}\). In order to have the prism fillers be related to the cubes, we expect the coherence of \autoref{defn:squares-over} to appear in the top triangle of \(T_A^{n+1}\). With that in mind, we factor the top triangle as shown in \autoref{fig:top-factor}, into the coherence of \autoref{defn:squares-over} and a technical adjustment described in \autoref{defn:top-triangle-inner}. The unfolded definition of \(\PhiDiag{n}\) is highlighted in blue. To complete \(T_A^{n+1}\), we fill the vertical square by applying the inverse of the right computation rule of \(t_B^{n+1}\).

\begin{constr}
\href{https://archive.vojtechstep.eu/zigzag-construction/synthetic-homotopy-theory.dependent-functoriality-sequential-colimits.html\#general-definition-in-dependent-squares}{\ExternalLink}
Given a function \(f : A \to B\), a family of equivalences \(e\{a : A\}: P(a) \to Q(f(a))\), a path \(r : x = y\) in \(A\), and a path \(t : z = f(y)\) in \(B\), there is a homotopy
\begin{center}
\begin{tikzcd}
  P(x) \arrow[r, "r \tr"] \arrow[d, "e"'] &
  P(y) \\
  Q(f x) \arrow[d, "(t \concat \ap_f(r^{-1}))^{-1} \tr"'] \\
  Q(z) \arrow[r, "t \tr"'] &
  Q(f y) \arrow[uu, "e^{-1}"']
\end{tikzcd}
\end{center}
defined by path induction on \(r\) and \(t\). After induction, all the transports compute away, and the homotopy \(e^{-1} \comp e \htpy \id\) is filled by the proof that \(e^{-1}\) is a retraction of \(e\).
\label{defn:top-triangle-inner}
\end{constr}

In the formalization, we must again consider computational behavior. Since \(t_B^n\) and \(t_A^n\) are defined together by induction on \(n\), and the definition of \(t_B^{n + 1}\) does one more case split on \(n\), we end up with three cases for the induction, namely \(0\), \(1\) and \(n + 2\). As a consequence, \(\overline{t_A^{n + 1}}(\inr(s, \refl, p))\) does not have a uniform definition, because it is also defined by cases \(0\), \(1\) and \(n + 2\). But we rely on its definition when computing the coherence of \(t_B^{n + 1}\). If we try to case split on \(n\) again to handle the cases \(t_A^1\) and \(t_A^{n + 2}\) separately, we would end up defining everything in terms of the cases \(0\), \(1\), \(2\) and \(n + 3\), and encounter the same problem. Instead, during the definition of \(t_B^{n + 1}\) and \(t_A^{n + 1}\) we carry a proof that \(\overline{t_A^{n + 1}}(\inr(s, \refl, p))\) is identical to the expected composition of \(\PhiDiag{n}\) and \(t_B^{n+1}\).

Furthermore, during induction we need to compute with \(t_B^{n + 1}\) and \(t_A^{n + 1}\) as maps defined by the dependent universal property, meaning that we need to carry around their defining dependent cocones. Rather than defining together the maps, the dependent cocones, and proofs that the maps are defined by the respective dependent cocones, we prefer to construct only the dependent cocones during induction, materializing their induced maps \(t_B^{n + 1}\) and \(t_A^{n + 1}\) only when necessary.

\begin{defi}
Given a natural number \(n\), the type of \textbf{section cocones}~\href{https://archive.vojtechstep.eu/zigzag-construction/synthetic-homotopy-theory.zigzag-construction-identity-type-pushouts.html\#the-type-of-section-cocones}{\ExternalLink} at stage \(n\) is the type of triples \((d_B^n, d_A^n, R^n)\), where \(d_B^n (b : B) : \depCoconePO(\PB{n+1}(b), Q_{\Sigma B}^{n+1})\) and \(d_A^n (a : A) : \depCoconePO(\PA{n+1}(a), Q_{\Sigma A}^{n+1})\) are families of dependent cocones over the cocones \(\PB{n + 1}(b)\) and \(\PA{n + 1}(a)\), respectively, and \(R^n(s)\) is an identification between the element \(\PhiDiag{n}(\depCogapPO(d_B^n(g s), p))\) and the right map of \(d_A^n(f s)\) applied to \((s, \refl, p)\), for all \(p : \PB{n + 1}(g s)\).
\end{defi}

\begin{constr}
\href{https://archive.vojtechstep.eu/zigzag-construction/synthetic-homotopy-theory.zigzag-construction-identity-type-pushouts.html\#section-cocones}{\ExternalLink}
For any natural number \(n\), construct a section cocone at stage \(n\)\label{dA}\label{dB}.

Begin by case splitting on \(n\). Define the left map \(d_B^0(b, \inl(p)) : Q_B^1(\inclB{0}(p))\) by \(\exf(p)\), the right map \(d_B^0(g s, \inr(s, \refl, p : \PA{0}(a_0))) : Q_B^1(p \concats{0} s)\) by induction on \(p : a_0 = f s\), returning \(Q_{\Sigma S}^0(q_0)\). Define the coherence \(d_B^0(g s, \glue(s, \refl, p))\) by \(\exf(p)\) as well. For \(d_A^0\), define the left map \(d_A^0(a, \inl(p)) : Q_A^1(\inl(p))\) by induction on \(p : a_0 = a\), mapping \(\refl\) to \(\kapA{0}(\refl) \tr q_0\). To define the right map, replace the \(t_B^1\{g s\}\) from the informal description by \(\depCogapPO(d_B^0(g s))\), making it out to be \(d_A^0(f s, \inr(s, \refl, p)) \defeq \PhiDiag{0}(\depCogapPO(d_B^0(f s), p))\). Then the coherence
\begin{equation*}
d_A^0(f s, \glue(s, \refl, \refl)) : \glueA{0}(\refl) \tr d_A^0(f s, \inclA{0}(\refl)) = d_A^0(f s, (\refl \concats{0} s) \concatinvs{1} \overline{s})
\end{equation*}
is defined exactly as in the informal description, by pasting a square and a dependent triangle. The identification \(R_0\) is inhabited by \(\refl\).

In the successor case we proceed completely analogously, following the informal description, replacing \(t_B^{n + 1}\{b\}\) by \(\depCogapPO(d_B^n(b))\) and \(t_A^{n+1}\{a\}\) by \(\depCogapPO(d_A^n(a))\). The only other difference is when defining the coherence \(d_B^{n+1}(g s, \glue(s, \refl, p))\) --- when filling the vertical square, we first need to apply \(R_n(p)\), and only then use the right computation rule of \(\depCogapPO(d_A^n(f s))\).
\end{constr}

We now have enough data to define the dependent map components of the section of \((Q_{\Sigma A}, Q_{\Sigma B}, Q_{\Sigma S})\).

\begin{constr}
\href{https://archive.vojtechstep.eu/zigzag-construction/synthetic-homotopy-theory.zigzag-construction-identity-type-pushouts.html\#induced-section-of-the-right-dependent-family}{\ExternalLink} Construct the map \(\tBinf \{b : B\} : (p : \PBinf(b)) \to Q_{\Sigma B}(p)\)\mlabel{tBinf} using the universal property of \(\PBinf(b)\) as the colimit of \(\PBdiag{\bullet + 1}(b)\), meaning we are allowed to provide the successor cases only. For the maps \(\tBn{n+1}(p : \PB{n+1}(b)): Q_{\Sigma B}^{n+1}(p)\)\mlabel{tBn} use \(\depCogapPO(\dB{n}(b))\). For the coherences \(K_B^{n+1}(p : \PB{n+1}(b)) : \kapB{n+1}(p) \tr t_B^{n+1}(p) = t_B^{n+2}(\inclB{n+1}(p))\), observe that the left computation rule of \(\tBn{n+2}\) gives us an identification \(\tBn{n+2}(\inclB{n+1}(p)) = \kapB{n+1}(p) \tr \tBn{n+1}(p)\), which may be inverted to define \(K_B^{n+1}\).
\end{constr}

\begin{constr}
\href{https://archive.vojtechstep.eu/zigzag-construction/synthetic-homotopy-theory.zigzag-construction-identity-type-pushouts.html\#induced-section-of-the-left-dependent-family}{\ExternalLink} Construct the map \(\tAinf \{a : A\} : (p : \PAinf(a)) \to Q_{\Sigma A}(p)\)\mlabel{tAinf} by induction on \(p\). Define the maps \(\tAn{n}(p : \PA{n}(a)) : Q_{\Sigma A}^n(p)\)\mlabel{tAn} by case analysis, taking \(\tAn{0}(\refl) \defeq q_0\) and \(\tAn{n+1} \defeq \depCogapPO(\dA{n}(a))\). For the coherences \(K_A^n(p : \PA{n}(a)) : \kapA{n}(p) \tr \tAn{n}(p) = \tAn{n+1}(\inclA{n}(p))\), again proceed by case analysis, and use the inverses of the left computation rules of \(\tAn{1}\) and \(\tAn{n+2}\) for the zero and successor cases, respectively.
\end{constr}

The final datum we need is the coherence square \(t_S^{\infty}\{s\}\) between \(\tAinf\{f s\}\) and \(\tBinf\{g s\}\), filling the right square of \autoref{fig:app-cube}. As suggested by the diagram, the square is constructed by applying \autoref{lem:preserves-cubes}. The first step is to construct the cubes as indicated, by pasting pairs of prisms at every index; the second step is to construct a homotopy between the top face obtained by pasting and the top face expected by the lemma.

\begin{figure}
\begin{center}
\begin{tikzcd}[column sep=small, row sep=small]
  Q_{\Sigma A}^0 \arrow[dr, "Q_{\Sigma S}^0" pos=0.7] \arrow[rr, "\kapA{0} \tr"]
  & & Q_{\Sigma A}^1 \arrow[dr, "Q_{\Sigma S}^1" pos=0.7] \arrow[rr] \arrow[from=dd, "\tAn{1}", very near start]
  &[-5pt] &[-5pt] \cdots
  &[-5pt] &[-5pt]
  Q_{\Sigma A}
  \arrow[dr, "Q_{\Sigma S}", "\simeq"']
  \\
  & Q_{\Sigma B}^1 \arrow[rr, crossing over, "\kapB{1} \tr"{pos=0.2}, "/"{marking, near end}] \arrow[ur, "\PhiDiag{0}"]
  & & Q_{\Sigma B}^2 \arrow[rr]
  & & \cdots
  & &
  Q_{\Sigma B} \\
  \PA{0}(f s) \arrow[uu, "\tAn{0}"] \arrow[dr, "\blank \concats{0} s"'] \arrow[rr, "\inclA{0}" very near start]
  & & \PA{1}(f s) \arrow[rr] \arrow[dr, "\blank \concats{1} s"', near start]
  & & \cdots
  & &
   \PAinf(f s) \arrow[dr, "\blank \concatinf s"'] \arrow[uu, "\tAinf"]
  \\
  & \PB{1}(g s) \arrow[uu, "\tBn{1}"', near start, crossing over] \arrow[rr, "\inclB{1}"'] \arrow[ur, "\blank \concatinvs{1} \overline{s}"' near start]
  & & \PB{2}(g s) \arrow[uu, "\tBn{2}"', near start, crossing over] \arrow[rr]
  & & \cdots
  & &
  \PBinf(g s) \arrow[uu, "\tBinf"']
\end{tikzcd}
\end{center}
\caption{\label{fig:app-cube}Strategy for defining the coherence \(\tSinf\).}
\end{figure}

\begin{figure}
\begin{tikzcd}[column sep=small, row sep=small]
  Q_{\Sigma A}^n \arrow[rr, "/"{marking, near end}, "\kapA{n}\tr"]
  \arrow[dr, "Q_{\Sigma S}^n" pos=0.7]
  & & Q_{\Sigma A}^{n+1} \\
  & Q_{\Sigma B}^{n+1} \arrow[ur, "\PhiDiag{n}"'] \\
  \PA{n}(f s) \arrow[rr, "\inclA{n}" very near start] \arrow[uu, "\tAn{n}"]
  \arrow[dr, "\blank \concats{n} s"']
  & & \PA{n+1}(f s) \arrow[uu, "\tAn{n+1}"']  \\
  & \PB{n+1}(g s) \arrow[ur, "\blank \concatinvs{n+1} \overline{s}"'] \arrow[uu, crossing over, "\tBn{n+1}"', near end]
\end{tikzcd}
\hspace{2em}
\begin{tikzcd}[column sep=small, row sep=small]
  & Q_{\Sigma A}^{n+1} \arrow[dr, "Q_{\Sigma S}^{n+1}"] \arrow[from=dd, "\tAn{n+1}"', very near start] \\
  Q_{\Sigma B}^{n+1} \arrow[rr, crossing over, "/"{marking, near end}, "\kapB{n+1}\tr" near start]
  \arrow[ur, "\PhiDiag{n}"]
  & & Q_{\Sigma B}^{n+2} \\
  & \PA{n+1}(f s) \arrow[dr, near start, "\blank \concats{n+1} s"'] \\
  \PB{n+1}(g s) \arrow[rr, "\inclB{n+1}"'] \arrow[uu, "\tBn{n+1}"]
  \arrow[ur, near end, "\blank \concatinvs{n+1} \overline{s}"]
  & & \PB{n+2}(g s) \arrow[uu, "\tBn{n+2}"']
\end{tikzcd}
\caption{\label{eqn:coherence-prisms}Dependent prisms to be pasted into a coherence cube.}
\end{figure}

The two prisms we want to paste are depicted in \autoref{eqn:coherence-prisms}. In the left prism, the bottom triangle is the \(\glueA{n}\) path constructor, the back square is the inverse of the left computation rule of \(\tAn{n+1}\), the right square is the inverse of the right computation rule of \(\tAn{n+1}\), and the composition of the top triangle and left square is the coherence datum of \(\tAn{n+1}\). After transposing the inverted faces to the other side of the equality in the type of the prism, we can fill the coherence by the computation rule of \(\tAn{n+1}\) on the path constructor.

A similar argument fills the right prism --- the bottom triangle is \(\glueB{n+1}\), the front and right squares are computation rules of \(\tBn{n+2}\), and the left square and top triangle are the coherence datum of \(\tBn{n+2}\). Filling this prism also requires mechanically adjusting the outline.

Once the two prisms are constructed, we use the fact that they share the diagonal square, so we may glue them together. This is achieved by algebraic manipulation of dependent paths and cylinders of sections, and mirrors the way the triangles are glued at the bottom to form the base square in \autoref{constr:zigzag-hom}.

The resulting cube has the correct vertices, edges, vertical squares, and bottom square, but the top square we get by composition is not the one required by \autoref{lem:preserves-cubes}. Fortunately we defined the top triangles in a way that makes it possible to show that the two squares are homotopic. The proof is constructed by appropriate abstraction followed by path induction, and is not very insightful, so it is omitted. Details are available in the formalization.

\begin{lem}
\href{https://archive.vojtechstep.eu/zigzag-construction/synthetic-homotopy-theory.zigzag-construction-identity-type-pushouts.html\#realigning-the-top-face-of-the-cubes}{\ExternalLink}
For every element \(p : \PA{n}(f s)\) there is a homotopy between the square
\begin{center}
\begin{tikzcd}
  Q_{\Sigma A}^n(p) \arrow[r, "\kapA{n}(p) \tr"]
  \arrow[d, "Q_{\Sigma S}^n"']
  &[7em] Q_{\Sigma A}^{n+1}(\inclA{n}(p))
  \arrow[d, "Q_{\Sigma S}^{n+1}"] \\
  Q_{\Sigma B}^{n+1}(p \concats{n} s) \arrow[r, "(H_n(p) \tr) \comp (\kapB{n+1}(p \concats{n} s) \tr)"']
  & Q_{\Sigma B}^{n+2}(\inclA{n}(p) \concatinvs{n+1} \overline{s})
\end{tikzcd}
\end{center}
from \autoref{defn:squares-over}, and the pasting of the two top dependent triangles in \autoref{eqn:coherence-prisms}.
\label{lemma:coherent-top}
\end{lem}

\begin{constr}
\href{https://archive.vojtechstep.eu/zigzag-construction/synthetic-homotopy-theory.zigzag-construction-identity-type-pushouts.html\#induced-square-of-sections}{\ExternalLink}
Construct the family of homotopies \(\tSinf\{s : S\}(p : \PAinf(f s)) : Q_{\Sigma S}(\tAinf(p)) = \tBinf(p \concatinf s)\)\label{tSinf} by application of \autoref{lem:preserves-cubes}. Use the right computation rules of \(\tBn{n+1}\) for the faces \(F_n\). Take the cubes to be the pasting of the prisms from \autoref{eqn:coherence-prisms}, with their top faces adjusted by \autoref{lemma:coherent-top}.
\label{defn:tS}
\end{constr}

\begin{thm}
\href{https://archive.vojtechstep.eu/zigzag-construction/synthetic-homotopy-theory.zigzag-construction-identity-type-pushouts.html\#the-zigzag-construction-is-an-identity-system}{\ExternalLink}
The descent data \((\PAinf, \PBinf, \blank \concatinf s)\) pointed at \(\reflinf\) is an identity~system.
\end{thm}

\begin{proof}
For arbitrary descent data \((Q_{\Sigma A}, Q_{\Sigma B}, Q_{\Sigma S})\) over the total span of the zigzag descent data pointed at \(q_0\), a section is given by \((\tAinf, \tBinf, \tSinf)\) constructed above. The equality \(\tAinf(\reflinf) = q_0\) holds by unfolding the left side to \(\tAinf(\inA{0}(\refl))\), and using the left computation rule to get \(\tAn{0}(\refl)\), which is defined to be \(q_0\).
\end{proof}

\begin{cor}
\href{https://archive.vojtechstep.eu/zigzag-construction/synthetic-homotopy-theory.zigzag-construction-identity-type-pushouts.html\#the-equivalences-with-path-spaces-of-pushouts}{\ExternalLink}
There are equivalences \(e_A\{a : A\} : (\inl(a_0) = \inl(a)) \simeq \PAinf(a)\) and \(e_B\{b : B\} : (\inl(a_0) = \inr(b)) \simeq \PBinf(b)\) such that for all \(p : (\inl(a_0) = \inl(f s))\) there is an equality \(e_B(p \concat H s) = e_A(p) \concatinf s\).
\end{cor}
\section{Conclusion and Related work}
\label{sec:org73f5cb7}

We have presented an encoding of the zigzag construction in a proof assistant, and gave a formal proof that it characterizes the path spaces of pushouts. The exposition illustrated subtleties of the inductive constructions, as a naïve transcription from paper would not have the necessary computational properties. The author was not able to leverage Agda's mutually recursive definitions, as seen in the formalization of the James construction. We pointed out some friction points with translating symbolic and diagrammatic reasoning to Agda. Those often necessitate verbose technical lemmas to manipulate paths in axiomatic HoTT. The author nonetheless found drawing diagrams to be invaluable for designing proofs in synthetic homotopy theory, such as \autoref{lem:preserves-cubes} and \autoref{defn:tS}. It is possible that a cubical type theory, such as the one implemented in Cubical Agda based on Cohen et al. \cite{CCHM}, could improve this specific aspect of the formalization, but we did not investigate this direction. Similarly, it could be worthwhile to attempt a formalization in a framework for simplicial type theory \cite{RS17,GWB26}, which has the potential to express more natural proofs of functoriality principles via synthetic morphisms. The author expects that a simple port of the current development would not bring much insight, so idiomatic proofs for the particular type theories should be developed first.

The author is aware of two previous attempts to formalize the zigzag construction. The first one is his previous work on the topic \cite{Step24}. It defines the square \autoref{defn:tS} using the universal property, but postulates the necessary coherences, since the unstructured goal one obtains by simply applying the universal property is hard to fill without the problem decomposition insight provided by this paper. The second, by Connors and Thorbjørnsen \cite{CT25}, based on the Coq-HoTT library \cite{Coq-HoTT}, is missing the same coherences.

There is a now an alternative description of the zigzag construction proposed by Wärn \cite{War24}, which presents it in more categorical than type theoretical terms, \eg by using types-over and pullback squares instead of type families and fiberwise equivalences. While the author did not attempt its formalization, he does not expect it to be necessarily easier than the one presented. The author attempted proving \autoref{lem:preserves-cubes} by taking the unstraightened perspective and defining sections of dependent diagrams to be morphisms with a retraction. Such a framework sidesteps talking about sections and fiberwise morphisms, but it raises the depth of path spaces we need to care about --- a cube of sections would be a pair of commuting cubes such that their composition is homotopic to a certain trivial cube. Type-checking the definition of cube composition had a noticeable performance impact, and defining a homotopy of cubes became practically intractable.
\section*{Acknowledgments}
\label{sec:org8a067e4}
I am grateful to Egbert Rijke for introducing me to synthetic homotopy theory, and for mentoring me through most of the time I spent on this formalization. I would like to thank Cyril Cohen and Assia Mahboubi, as well as anonymous reviewers, for providing feedback on earlier versions of this paper, which substantially improved the exposition.

\bibliographystyle{alphaurl}
\bibliography{./bibliography}
\end{document}